\documentclass[11pt,letterpaper]{article}

\usepackage[affil-it]{authblk}
\usepackage[utf8]{inputenc}
\usepackage{amsmath}
\usepackage{amsfonts}
\usepackage{amssymb}
\usepackage{amsthm}
\usepackage{graphicx}
\usepackage{subcaption} 
\usepackage[left=2cm,right=2cm,top=2cm,bottom=2cm]{geometry}
\usepackage{enumerate}
\usepackage{appendix}
\usepackage{mathtools}
\usepackage{xcolor}
\usepackage{hyperref}
\usepackage{todonotes}
\usepackage{cleveref}

\newtheorem{theorem}{Theorem}[section]
\newtheorem{lemma}[theorem]{Lemma}
\newtheorem{prop}[theorem]{Proposition}

\theoremstyle{definition}
\newtheorem{definition}[theorem]{Definition}

\theoremstyle{remark}
\newtheorem{remark}[theorem]{Remark}

\newcommand{\pare}[1]{\left( #1 \right)}

\newcommand{\system}[1]{\left\{ #1 \right.}

\usepackage[maxbibnames=99,backend=bibtex,style=numeric-comp]{biblatex}

\numberwithin{equation}{section}

\date{}
\begin{document}

\title{\bf Gevrey-class-$3$ regularity of the linearised hyperbolic Prandtl system on a strip}

\author{Francesco De Anna$\,^1$, Joshua Kortum$\,^2$,   Stefano Scrobogna$\,^3$}
\affil{\small
    $\,^1$ Institute of Mathematics, University of Würzburg, Germany\\
    email: francesco.deanna@mathematik.uni-wuerzburg.de\\

    \vspace{0.3cm}
    $\,^2$ Institute of Mathematics, University of Würzburg, Germany \\
    email: joshua.kortum@mathematik.uni-wuerzburg.de\\
    
     \vspace{0.3cm}
    $\,^3$ Dipartimento di Matematica e Geoscienze, Università degli Studi di Trieste, Trieste, Italy, \\
    email: stefano.scrobogna@units.it
}

\maketitle

\begin{abstract}
 \noindent  
In the present paper, we address a physically-meaningful extension of the linearised Prandtl equations around a shear flow. Without any structural assumption, it is well-known that the optimal regularity of Prandtl is given by the class Gevrey 2 along the horizontal direction. The goal of this paper is to overcome this barrier, by dealing with the linearisation of the so-called \textit{hyperbolic Prandtl equations} in a strip domain. We prove that the local well-posedness around a general shear flow $U_{\rm sh}\in W^{3, \infty}(0,1)$ holds true, with solutions that are Gevrey class 3 in the horizontal direction.
\end{abstract}

\section{Introduction}

The main concern of this paper is to set up a rigorous well-posedness theory for the following extension of the linearised and reduced Prandtl equations on a thin strip:
\begin{equation}
    \label{LHP}
    \system{
    \begin{alignedat}{4}
    & ( \partial_t +1)\big(
    \partial_t u  + U_{{\rm sh}}  \partial_x u + v \ U'_{{\rm sh}}\big) - \partial_{y}^2 u =0,
    \qquad (t,x,y)\in \,
    &&(0,T) \times \mathbb{T}\times (0,1),\\
    & \partial_x u + \partial_y v =0
    &&(0,T) \times \mathbb{T}\times (0,1),
    \\
    & \left. \pare{u, u_t}\right|_{t=0} =\pare{u_{\rm in}, u_{t,\rm in}}
    &&\hspace{1.425cm}\mathbb{T}\times (0,1),
    \\
    & \left. u\right|_{y=0}=0,\quad v_{|y = 0} = 0
     &&(0,T) \times \mathbb{T}.
    \end{alignedat}
    } 
\end{equation}
In this system, the unknown is the horizontal component $u= u(t,x,y)$ of the velocity field $(u,v)^T: (0,T) \times\mathbb{T}\times (0,1) \to \mathbb{R}^2 $, while the vertical component $v = v(t,x,y)$ is determined by the divergence-free condition $\partial_x u + \partial_y v = 0$ and by the Dirichlet relation $v_{|y = 0} = 0$, which formally imply
\begin{equation*}
    v(t,x,y) = -\partial_x \int_0^y u(t,x,z) dz,\qquad (t,x,y) \in 
    (0,T)\times \mathbb T \times (0,1).
\end{equation*}
The function $U_{\rm sh} = U_{\rm sh}(y)$ depends uniquely upon the vertical variable $y\in (0,1)$ and describes a shear flow $(U_{\rm sh}(y), 0)^T$, around which the original equations have been linearised. 
System \eqref{LHP} arises indeed from a meaningful extension of the classical Prandtl equations, known as \textit{hyperbolic Prandtl equations}\cite{PZ2022}. With our analysis, we aim to show how System~\eqref{LHP} might be particularly desirable, in order to overcome certain analytic barriers that are typical of the classical Prandtl theory. More precisely, System \eqref{LHP} is amenable in terms of solutions that have regularity Gevrey-class 3 along the vertical variable $x\in \mathbb T$, overcoming therefore the well-known barrier of Gevrey-class 2 of the Prandtl theory. Details about this statement and our main result are presented starting from \Cref{sec:preliminaries-main-result}. 
First, we shall provide some background on the origin of this model.

\subsection{The Prandtl equations and the barrier of Gevrey-class 2}\label{sec:the-barrier-of-Gevrey-2}
\label{sec:barrier-Gev2}
 
\noindent
In order to understand the major characteristics of \eqref{LHP}, we shall briefly overview the original model of Prandtl, that was introduced during a 10 minutes presentation of the 1904 Third International Mathematics Congress in Heidelberg \cite{Tollmien1961}. Such short presentation has scientifically impacted many disciplines, so much that nowadays the field of aerodynamics is still shaped by his fundamental idea: in order to describe the inviscid limit of an incompressible fluid in a region close to a solid surface (where dissipative forces are predominant), one would rather consider the velocity of the fluid in terms of rescaled variables, which concentrate the dissipative effects in a thin region close to the boundary. 
This procedure provides a ``split'' in the behaviour of the flow:
\begin{itemize}
    \item within the bulk,  the hydrodynamics is dominated by the incompressible Euler equations with no-penetration boundary conditions,
    \item on a neighbourhood of the boundary, a corrective term (the so-called boundary layer) provided by the Prandtl equations restores the natural no-slip boundary conditions for viscous flows. 
\end{itemize}
One of the simplest forms of the (non-linear) Prandtl equations in two dimensions is given by
\begin{equation}
    \label{Prandtl}
    \system{
    \begin{alignedat}{4}
    &  
    \partial_t u  +  u  \partial_x u + v \partial_y u  - \partial_{y}^2 u =\partial_t u^{E}+ u^E\partial_x u^E ,
    \qquad (t,x,y)\in 
    &&(0,T) \times \mathbb{X}\times (0,+\infty),\\
    & \partial_x u + \partial_y v =0
    &&(0,T) \times \mathbb{X}\times (0,+\infty),
    \\
    & \left. \pare{u,v}\right|_{y=0}=0\quad 
    \lim_{y\to+\infty} u = u^E
     &&(0,T) \times \mathbb{X},\\
    &u_{t|= 0} = u_{\rm in}
    &&\hspace{1.4cm} \mathbb{X}\times (0,+\infty),
    \end{alignedat}
    } 
\end{equation}
where $x\in \mathbb X$ describes the (local) arc-length parametrisation of the solid surface (usually in the mathematical community $\mathbb{X} = \mathbb T$ or $\mathbb X = \mathbb R$), while $u^E= u^E(t,x)$ is determined by the solution of the Euler equation in the bulk of the flow, when approaching the boundary. 

\noindent
The analysis of \eqref{Prandtl} has received from the mathematical community numerous investigations during the past decades. Although the Prandtl equations are classical, their applications are rather narrowed because of the particular unstable nature of the underlying solutions. 
These instabilities are nowadays moderately well understood and relate mainly to separation phenomena (appearance of reversed flow in the boundary layers). 

\noindent 
The first rigorous mathematical study addressing the well-posedness of the Prandtl equations \eqref{Prandtl} was performed in the book of Oleinik and Samokhin \cite{Oleinik1999} in the case of so-called monotonic initial data (namely initial velocity $u_{\rm in}$ in \eqref{Prandtl}, satisfying $\partial_y u_{\rm in} >0$). Roughly speaking, the Olenik's monotonicity prevents the mentioned flow separation, at least locally in time. This allows to recast the velocity field through a meaningful transformation (known as Crocco transformation), providing a solid ground to the local-in-time well posedness of \eqref{Prandtl} within function spaces typical of hydrodynamics, such as Sobolev ones. For more details on the Olenik's monotonicity, we refer the reader to the more recent result \cite{AWXY2015},  in which the authors construct local-in-time solutions via a Nash-Moser argument. See also \cite{MW2015} for a proof performed purely by energy methods. 

\noindent 
For initial data lacking monotonicity, the well-posedness becomes much more involved and one has to consider function spaces that control infinite derivatives of the solutions. This was addressed in the celebrated result \cite{SC1998} of Caflisch and Sammartino,  where the authors dealt with non-monotonic initial data $u_{\rm in}$ that are analytic in the variable $x\in \mathbb X = \mathbb R$. In the framework of a periodic variable $x\in \mathbb X = \mathbb T$, analytic initial data can be easily understood through their Fourier series, under a strong localisation of the frequencies:
\begin{equation*}
    u_{\rm in} (x,y) = \sum_{k\in \mathbb Z} u_{\rm in, k}(y) e^{i k x},
\end{equation*}
where the modes $u_{\rm in, k}(y)$ decays exponentially as $u_{\rm in, k}(y)\sim e^{-a|k|}$, for some radius of analyticity $a>0$. This type of initial data are however extremely regular and with reduced applications to real phenomena. For this reason, an increasing number of works were devoted to relax this framework.

\noindent
The first breakthrough was provided in \cite{GM2015} by G\'erard-Varet and Masmoudi, where the authors showed that the Prandtl system is actually locally well-posed for data that are Gevrey-class $7/4$ in the $x$-variable. Roughly speaking, an initial data $ u_{\rm in}$ is Gevrey-class $m$ along $x\in \mathbb T$, with $m>1$, if the modes $u_{\rm in, k}(y)$ decay exponentially as $u_{\rm in, k}(y)\sim e^{-a|k|^{1/m}}$, for a suitable radius $a>0$. 
Already in \cite{GM2015}, however, the authors remarked that the Gevrey-class $7/4$ was unlikely to be optimal and that further insights from numerics suggested rather a threshold of Gevrey-class $2$ (i.e.~$u_{\rm in, k}(y)\sim e^{-a|k|^{1/2}}$).

\noindent 
Eventually, this remark was mathematically formalised and a first result in this direction was attained in \cite{LY2020}, assuming that the velocity $(u,v)^T$ in \eqref{Prandtl}  is a small perturbation of a suitable shear flow, which satisfies a non-degenerate condition (for details cf.~Assumption~1.1 in \cite{LY2020}). 

\noindent
The breakthrough of Gevrey 2 was however achieved by Gerard-Varet and Dietert in \cite{MR3925144}, where the authors developed a robust local well-posedness theory without any structural assumption on the flow (such as monotonicity or critical points). Their result was based on a meaningful change of state variable, from which the present work takes substantial inspiration. 

\noindent 
From the work of Gerard-Varet and Dietert followed a variety of questions, in particular to determine whether the Gevrey-class 2 was optimal for the well-posedness of System~\eqref{Prandtl} or if further insights would have led to weaker regularities. 
Surprisingly, Gerard-Varet and Dormy overturned any possibility of improvement, providing indeed a negative answer to this open problem. In their seminal result \cite{GD2010}, the authors showed that already at the level of the linearised equations around a shear flow $(U_{\rm sh}(y), 0)$, namely replacing the first equation of \eqref{Prandtl} with 
\begin{equation} \label{LP}
    \partial_t u  + U_{{\rm sh}}  \partial_x u + v \ U'_{{\rm sh}} - \partial_{y}^2 u =0,
    \qquad (t,x,y)\in \,
    (0,T) \times \mathbb{T}\times (0,+\infty),
\end{equation}
the linear propagator of regularity is unbounded in Gevrey-class higher than 2. Roughly speaking, the authors showed the existence of solutions, whose modes in the frequencies $k\in \mathbb Z$ experience an exponential growth with rate $|k|^{1/2}$. In general, this growth could be counteracted only by Gevrey-$2$ initial data, precluding any room for improvement.  In other words, this was the first encounter with the barrier of Gevrey 2: the linearised Prandtl equation is ill-posed within any  larger setting. 

\noindent 
We refer the reader to the works \cite{DM2019,GLSS2009, GGN2016,GGN2015}, as well, which concern further instabilities of the Prandtl equations.  

\subsection{The hydrostatic approximation}
\label{sec:hydrostatic-approximation}
When the vertical variable $y$ is bounded, for instance with $y\in (0,1)$ (as in our System \eqref{LHP}), a different type of equations has adequately found a mathematical relevance, namely the so-called \textit{hydrostatic approximation of Navier-Stokes/Prandtl}. In two dimensions, these equations are a reminiscent of Prandtl and take the form 
\begin{equation}
    \label{hydrostatic-approximation}
    \system{
    \begin{alignedat}{4}
    &  
    \partial_t u  +  u  \partial_x u + v \partial_y u  - \partial_{y}^2 u +\partial_x p=0,
    \qquad (t,x,y)\in 
    &&(0,T) \times \mathbb{X}\times (0,1),\\
     &  
    \partial_y p =0,
    &&(0,T) \times \mathbb{X}\times (0,1),\\
    & \partial_x u + \partial_y v =0
    &&(0,T) \times \mathbb{X}\times (0,1),
    \\
    & \left. \pare{u,v}\right|_{y=0,1}=0
     &&(0,T) \times \mathbb{X},\\
    &u_{t|= 0} = u_{\rm in}
    &&\hspace{1.4cm} \mathbb{X}\times (0,1).
    \end{alignedat}
    } 
\end{equation}
This model is significant in several phenomena of atmospheric science and can be derived from the so-called primitive equations. Beside the vertical domain, System \eqref{hydrostatic-approximation} inherently differs from \eqref{Prandtl} in the boundary conditions of $v$. In \eqref{hydrostatic-approximation} $v$ is null in both $y= 0,1$, whereas in \eqref{Prandtl} $v$ has homogeneous condition only in $y = 0$ (without any assumption for $y\to +\infty$). Because of this, the pressure $p$ in the hydrostatic approximation \eqref{hydrostatic-approximation} is non-trivial and can be interpreted as a Lagrangian multiplier associated to the constraint $v_{|y= 1}= 0$. We refer the reader to the works \cite{Brenier1999,Brenier2003,GMV2020,Grenier1999,MW2012,MW2015,Wong2015,PZZ2020} and as well the interesting result in \cite{Renardy2009} in which the author proves that, contrarily to what happens for the Prandtl equations, the presence of an inflexion point may trigger high-frequencies instabilities in the linearization of \cref{hydrostatic-approximation} around a shear flow, i.e. \cref{hydrostatic-approximation} in which the first equation is substituted by \cref{LP}. We want to highlight that, to the best of our knowledge, the best regularity result for \cref{hydrostatic-approximation} is provided in \cite{GMV2020,WWZ2021} for $9/8$--Gevrey data {\it under an additional convexity assumption}. Hence the optimal stability vs. instability question is still an open question for the system \eqref{hydrostatic-approximation}, contrarily to what is known for the Prandtl system \eqref{Prandtl}.

\begin{remark}
This paper addresses the well-posedness of \eqref{LHP} within $y\in (0,1)$, nevertheless our intent is to provide insights about an extension of the Prandtl equations \eqref{Prandtl} (for which we know that the barrier is Gevrey 2) rather than the hydrostatic ones in \eqref{hydrostatic-approximation}. 
Dealing with the pressure and homogeneous Dirichlet conditions on $v$ is beyond our interest (certainly, with the pressure, the problem would be much more involved). We address a bounded vertical domain $y\in (0,1)$ uniquely for the sake of a clear presentation of our analysis. To the best of our knowledge, our work is indeed the first to overcome the barrier of Gevrey 2 for a meaningful extension of Prandtl. We infer that a similar result can be achieved in the classical domain $(t,x,y)\in (0,T)\times \mathbb T \times (0, \infty)$, making use of a related ansatz on function spaces with weighted norms in the vertical direction.
\end{remark}

\subsection{The Cattaneo's law on the hydrostatic approximation}\label{sec:cattaneo}
Besides the barrier of Gevrey~2, a more physical drawback of Systems \eqref{Prandtl} and \eqref{hydrostatic-approximation} can be found at the level of the Navier-Stokes equations (from which \eqref{Prandtl} and \eqref{hydrostatic-approximation} are indeed asymptotically derived), because of the so-called {\it infinite propagation speed} of the velocity field (any local variation of the velocity field perturbs immediately the flow in all the domain). To avoid this scenario (which may be occasionally unsatisfactory, especially in the hydrodynamics of fluids at large scale), a suitable hyperbolic extension of Navier-Stokes has found growth in popularity in the mathematical community (cf.~\cite{MR3942552, BNP2004, CHR2022, PR2007, PZ2022, MR3085225, RackeSaal2012}). At a first glance, this extension seems to introduce obstacles, for instance it enlarges the hydrostatic equations \eqref{hydrostatic-approximation} into
\begin{equation}
\label{eq:Prandtl_hyperbolic}
    \system{
    \begin{alignedat}{4}
    &  
    \big( 
        \tau \partial_t+1
    \big)
    \big(
    \partial_t u  +  u  \partial_x u 
    + v \partial_y u  
    \big)
    - \partial_{y}^2 u +\partial_x p =
    0,
    \quad  
    &&(0,T) \times \mathbb{X}\times (0,1),\\
     & \partial_y p =0
    &&(0,T) \times \mathbb{X}\times (0,1),
    \\
    & \partial_x u + \partial_y v =0
    &&(0,T) \times \mathbb{X}\times (0,1),
    \\
     & \left. \pare{u,v}\right|_{y=0,1}=0
     &&(0,T) \times \mathbb{X},\\
    &\left. \pare{u, u_t}\right|_{t=0} =\pare{u_{\rm in}, u_{t,\rm in}}
    &&\hspace{1.4cm} \mathbb{X}\times (0,1),
    \end{alignedat}
    },
\end{equation}
where $\tau>0$ is a meaningful parameter, konwn as relaxation time. Once more, the pressure $p$ in \eqref{eq:Prandtl_hyperbolic} is uniquely due to $v_{|y= 1} = 0$ and would vanish when relaxing this constraint (as in our model \eqref{LHP}).

\noindent
System \eqref{eq:Prandtl_hyperbolic} arises (at least formally) from the inviscid limit of the Navier-Stokes equations, whose Cauchy stress tensor is ``delayed'' through a first-order Taylor expansion:
\begin{equation*}
    \mathbb{S}(t+ \tau, \cdot) \approx  \mathbb{S}(t, \cdot)+\tau \partial_t  \mathbb{S}(t, \cdot) =\nu \frac{\nabla u(t,\cdot)+\nabla u(t,\cdot)^T}{2}
\end{equation*}
(we refer to \cite{BNP2004} for more details). This relation was introduced in fluid-dynamics by Carrassi and Morro~\cite{CARRASSIMORRO},  inspired by the celebrated work of Cattaneo \cite{Cattaneo1949, Cattaneo1958} on heat diffusion.


\noindent 
Despite its relevance, the well-posedness theory of System \eqref{eq:Prandtl_hyperbolic} is unfortunately much less understood. In \cite{PZ2022}, the authors considered $\tau = 1$ and neglected the term $\partial_t(u \partial_x u + v\partial_y u)$ in the first equation. By exploiting a similar technique as the one used in \cite{PZ2021}, they showed that small initial data in Gevrey 2 generate global-in-time solutions. Moreover, in this framework, they justified the asymptotic limit of the Navier-Stokes under Cattaneo's law towards the solutions of \eqref{eq:Prandtl_hyperbolic}, when the viscosity is vanishing. 

\smallskip
\noindent
The main goal of this paper is to show that one can potentially overcome the barrier of Gevrey 2, when dealing with the relevant extensions \eqref{eq:Prandtl_hyperbolic} of the classical Prandtl equation. We indicate accurately this principle on the linearised equation \eqref{LHP} of System \eqref{eq:Prandtl_hyperbolic} around a shear flow $(U_{\rm sh}(y), 0)$, when the constraint $v_{|y=1}= 0$ is relaxed (hence no pressure is involded, as for Prandtl). 
We establish that this model is indeed well-posed locally-in-time, when the initial data are Gevrey-class 3 in $x$ (thus less regular than Gevrey 2) and Sobolev in $y$. We provide also some remarks about the non-linear system in Section~\ref{sec:remarks-on-the-nonlinear-system}.

\subsection{Preliminaries and statement of the main result}\label{sec:preliminaries-main-result}

\noindent 
To formalise our statement, we shall briefly recall the formal definition of Gevrey functions, that we will use throughout our analysis.
\begin{definition}\label{def:Gevrey-regularity}
Let $\sigma>0$ and $m\geq 1$. We define the Banach space $\mathcal{G}^m_{\sigma, x}L^2_y=\mathcal{G}^m_{\sigma}(\mathbb T, L^2(0,1))$ (resp.~$\mathcal{G}^m_{\sigma, x}H^1_{0,y}= \mathcal{G}^m_{\sigma}(\mathbb T, H^1_0(0,1))$) as all integrable functions $f\in L^1(\mathbb T\times (0,1))$ satisfying:
\begin{itemize}
    \item Each coefficient $f_k:(0,1)\to \mathbb R$ of the Fourier transform in the $x$-variable
\begin{equation*}
    f_k(y) = \frac{1}{2\pi}\int_\mathbb{T}f(x,y)e^{-ikx}dx,\qquad y\in (0,1), 
\end{equation*}
belongs to $L^2(0,1)$ (resp.~$H^1_0(0,1)$). 

\item The sequences of norms $(\| f_k \|_{L^2})_{k\in \mathbb Z}$ (resp.~$(\| \partial_y f_k \|_{L^2})_{k\in \mathbb Z}$) decays exponentially as $e^{-\sigma|k|^{1/m}}$ at high frequencies. 
\end{itemize}
More precisely, $f\in L^1(\mathbb T\times (0,1))$ belongs to $\mathcal{G}^m_{\sigma, x}L^2_y$ (resp.~$\mathcal{G}^m_{\sigma, x}H^1_{0,y}$), if the following norm is indeed finite:
\begin{equation}\label{norm:Gevrey-Sobolev}
    \| f \|_{\mathcal{G}^m_{\sigma, x}L^2_y}
    :=
    \Big\| e^{\sigma |k|^{\frac{1}{m}}}\| f_k\|_{L^2(0,1)}\Big\|_{\ell^\infty(\mathbb Z)}
    = 
    \sup_{k \in \mathbb Z}\bigg\{ e^{\sigma |k|^{\frac{1}{m}}}\bigg(\int_0^1 |f_k(y)|^2 dy\bigg)^\frac{1}{2} \bigg\}
    <+\infty,
\end{equation}
(resp.~$\| f \|_{\mathcal{G}^m_{\sigma, x}H^1_{0,y}} := \| \partial_y f \|_{\mathcal{G}^m_{\sigma, x}L^2_y}<+\infty$). 
\end{definition}
\noindent
Function spaces with Gevrey regularity are rather standard, especially in the mathematical treatment of the Prandtl equations. Indeed, by strongly localising the frequencies, one copes with the major instabilities of the underlying solutions. When $f= f(x)$ depends uniquely upon $x\in \mathbb T$, however, the definition of the Gevrey norm in \eqref{norm:Gevrey-Sobolev} may vary in terms of the preferred analytical tools. Among the most relevant norms, we mention for instance 
\begin{equation*}
   \bigg( 
   \sum_{k \in \mathbb Z} 
    e^{2\sigma |k|^{\frac{1}{m}}}|f_k|^2
   \bigg)^\frac{1}{2},
   \quad
   \sup_{n \in \mathbb N} 
   \bigg\{
   \frac{\sigma^{n}}{(n!)^{m}}
    \|\partial_x^n f\|_{L^\infty(\mathbb T)}
   \bigg\},
\end{equation*}
which are somehow equivalent to \eqref{norm:Gevrey-Sobolev}, for positive radii close to $\sigma$.
\begin{definition}
    Let $T$ denote a lifespan in $(0, +\infty]$ and $\eta:(0,T)\to (0, +\infty)$ be a positive continuous non-increasing function, representing the time-evolution of the Gevrey radius of regularity. We say that a function $f$ belongs to $L^p(0, T;\mathcal{G}^m_{\eta(t), x}L^2_{y}) $ (resp.~$L^p(0, T;\mathcal{G}^m_{\eta(t), x}H^1_{0,y})$), for a fixed  $1\leq p\leq \infty$, if 
    \begin{itemize}
        \item $f$ belongs to $L^p(0,T; L^1(\mathbb{T}\times (0,1)))$, 
        \item $f(t)$ belongs to $\mathcal{G}^m_{\eta(t), x}L^2_{y}$ (resp.~$\mathcal{G}^m_{\eta(t), x}H^1_{0,y}$), for almost any $t\in (0,T)$,
        \item the function $t\in (0,T)\to \|f(t) \|_{\mathcal{G}^m_{\eta(t), x}L^2_{y}}$ (resp.~$t\in (0,T)\to \|f(t) \|_{\mathcal{G}^m_{\eta(t), x}H^1_{0,y}}$) belongs to $L^p(0,T)$.
    \end{itemize}
\end{definition}

\noindent 
The function space being set up, the main goal of this paper is to establish the local-in-time existence of solutions for the linearised equation \eqref{LHP}, whose initial data $u_{{\rm in}}$ and  $u_{t,{\rm in}}$ are indeed Gevrey-class 3 in the horizontal variable, as described by \Cref{def:Gevrey-regularity}. 
\begin{theorem}\label{main-thm}
    Assume that the shear flow $y\in (0,1)\mapsto U_{{\rm sh}}(y)$ is in $W^{3, \infty}(0,1)$, while the initial data $u_{{\rm in}}, u_{t,{\rm in}}:\mathbb T \times (0,1) \to \mathbb R$ are in $\mathcal{G}^3_{\sigma, x}H^1_{0,y} $ and $\mathcal{G}^3_{\sigma, x}L^2_{y}$, respectively, for a positive $\sigma>0$. 
    Denoting by $T_\sigma>0$ the lifespan 
    \begin{equation}\label{def:Tsigma}
       T_\sigma :=
       \sup 
       \bigg\{t>0\quad\text{such that}\quad 
       \frac{\sigma}{8}
       -
       2^\frac{5}{6}
       \Big(
       \| U_{{\rm sh}}''' \|_{L^\infty}
       +
       2\| U_{{\rm sh}}'' \|_{L^\infty}
       \Big)^\frac{1}{3}e^\frac{t}{3}t
       > 0
       \bigg\}\in (0, +\infty]
   \end{equation}
   and by $\beta,\, \gamma : [0, T_\sigma)\to \mathbb (0, +\infty)$ the following time-dependent radii of Gevrey-class regularity 
   \begin{equation}\label{radii-of-Gevrey-beta-gamma}
   \begin{aligned}
        \beta(t) := 
        \frac{\sigma}{4}
        -
        2^\frac{5}{6}
        \Big(
        \| U_{{\rm sh}}''' \|_{L^\infty}
        +
        2\| U_{{\rm sh}}'' \|_{L^\infty}
        \Big)^\frac{1}{3}e^\frac{t}{3}t>0,\qquad \gamma(t) := \beta(t)-\frac{\sigma}{8}>0,
    \end{aligned}
   \end{equation}
   then the linearised system \eqref{LHP} admits a unique weak solution $u:[0,T_\sigma)\times \mathbb T \times (0,1)\to \mathbb R$ in the function space
   \begin{equation}\label{function-space-for-u-main-thm}
       u \in 
       L^\infty(0,T_\sigma;\mathcal{G}^3_{\beta(t), x}H^1_{0,y})
       \quad
       \text{with}
       \quad
       \partial_t u \in 
       L^\infty(0,T_\sigma;\mathcal{G}^3_{\gamma(t), x}L^2_{y}).
   \end{equation}
    Furthermore, the following estimate holds true at any time $t\in [0, T_\sigma)$:
    \begin{equation}\label{est:main-thm-inequality-Gevrey}
        \| u(t)             \|_{\mathcal{G}^3_{\beta(t), x}H^1_{0,y}} + 
        \| \partial_t u(t)  \|_{\mathcal{G}^3_{\gamma(t), x}L^2_{y}}
        \leq D_\sigma(U_{\rm sh})
        (1+t)^5e^t
        \Big(
            \| u_{\rm in}   \|_{\mathcal{G}^3_{\sigma, x}H^1_{0,y}} + 
            \| u_{t,\rm in} \|_{\mathcal{G}^3_{\sigma, x}L^2_{y}}
        \Big),
    \end{equation}
    for a suitable constant $D_\sigma(U_{\rm sh})>0$, which depends uniquely upon $\sigma>0$ and the $W^{3, \infty}$-norm of $U_{\rm sh}$.
\end{theorem}
\noindent 
Before presenting the major novelties and implications of this result, some remarks on the statement are here in order. 

\noindent 
The solutions of \Cref{main-thm} are weak because of the regularity of the initial data along the vertical direction $y\in (0,1)$. This regularity comes from the underlying a-priori energy of the System \eqref{LHP}
\begin{equation*}
    \frac{1}{2}
    \frac{d}{dt}
    \Big[
        \| (\partial_t +1) u \|_{L^2}^2+   \| \partial_y u \|_{L^2}^2 
    \Big]
    + 
    \| \partial_y u \|_{L^2}^2 
    = 0.
\end{equation*}
Although the solution $u$ and its derivative $\partial_t u$ are in $L^\infty$ with respect to the time variable, we write the estimate \eqref{est:main-thm-inequality-Gevrey} at any $t\in [0,T_\sigma)$ (instead of ``for a.e.~$t\in (0,T_\sigma)$''). This is due to the fact that both $u$ and $\partial_t u$ admit a continuous representative in the following space: for any finite $T\in (0, T_\sigma)$ and for any fixed radius $\mu \in (0, \beta(T))$ (resp.~$\mu \in (0, \gamma(T))$), $t\in [0,T]\to u(t)$ belongs to $\mathcal{C}([0, T],\mathcal{G}^3_{\mu, x}H^1_{0,y}) $ (resp.~$\partial_t u$ belongs to $\mathcal{C}([0, T],\mathcal{G}^3_{\mu, x}H^1_{0,y}) $). Writing directly the expression $u\in \mathcal{C}([0, T_\sigma),\mathcal{G}^3_{\beta(t), x}H^1_{0,y}) $ in 
\eqref{function-space-for-u-main-thm} would be somehow incomplete without a proper clarification, since the norm of $\mathcal{G}^3_{\beta(t), x}H^1_{0,y}$ does progress in time. Certainly one may define this continuity in terms of topology, however this would just complicate the readability of the paper. We maintain therefore the function space of \eqref{function-space-for-u-main-thm} and the estimate \eqref{est:main-thm-inequality-Gevrey} at any time $t\in [0, T_\sigma)$.

\noindent
The solutions provided by \Cref{main-thm} are (in general) only local-in-time, although system \eqref{LHP} is linear in its state variables. 
Indeed, as for the classical Prandtl equations, the high regularity of the initial data $u_{\rm in}$, $u_{t,\rm in}$  is persistently eroded by the term $v U_{\rm sh}'$ in \eqref{LHP} (and in our extension of Prandtl with Cattaneo also by  $\partial_t v U_{\rm sh}' $), together with the viscous operator $-\partial_y^2 u$ in $y\in (0,1)$ (cf.~\cite{GD2010}). This aspect is here reflected by the decaying of the Gevrey radii $\beta$ and $\gamma$ in \eqref{radii-of-Gevrey-beta-gamma}, as time $t>0$ approaches the lifespan $T_\sigma$. 

\noindent 
The radii $\beta(0)$ and $\gamma(0)$ in \eqref{radii-of-Gevrey-beta-gamma} at initial time $t= 0$
correspond to $\sigma/4$ and $\sigma/8$, respectively. One would rather expect them to coincide with the radius  $\sigma>0$ of the initial data. 
This gap is merely an artifact of our analysis, since we also aim to determine an explicit (and readable) form of the constant $D_s(U_{\rm sh})>0$ in \eqref{est:main-thm-inequality-Gevrey}. To this end, we make use of a given amount of the exponential decay in $e^{-\sigma|k|^{1/3}}$ (for instance the missing $e^{-3\sigma|k|^{1/3}/4}$ between $e^{-\gamma(0)|k|^{1/3}}$ and $e^{-\sigma|k|^{1/3}}$), in order to absorb the contribution of certain terms, which arise from $v U_{\rm sh}'$ and $\partial_t v U_{\rm sh}' $ (cf.~for instance \eqref{ineq:k-est-with-exp}). 
Accordingly, we may explicitely set $D_\sigma(U_{\rm sh})>0$ in  \eqref{est:main-thm-inequality-Gevrey} as 
\begin{equation}\label{def:Dsigma}
 D_\sigma(U_{\rm sh}):=10^{4}\max\{1, 12/\sigma\}^{15}(1+\|U_{\rm sh}\|_{L^\infty}+\|U_{\rm sh}'\|_{L^\infty}+\|U_{\rm sh}''\|_{L^\infty}+\|U_{\rm sh}'''\|_{L^\infty})^3 .
\end{equation}
This arrangement is certainly far from being sharp. We
may for instance build our Gevrey-class-3 solution with radii of regularity 
\begin{equation*}
    \beta_\varepsilon(t)
    =
    \gamma_\varepsilon(t) := 
    \sigma
    -\varepsilon
    -
    2^\frac{5}{6}
    \Big(
    \| U_{{\rm sh}}''' \|_{L^\infty}
    +
    2\| U_{{\rm sh}}'' \|_{L^\infty}
    \Big)^\frac{1}{3}e^\frac{t}{3}t,
\end{equation*}
for any small $\varepsilon>0$. This definition would nevertheless complicate the constant  $D_\sigma(U_{\rm sh})$ (behaving now also like $1/\varepsilon$). For the sake of simple presentation, we do not pursue this direction and we simply remark that as long as $u_{\rm in},u_{t,\rm in}$ are in $\mathcal{G}^m_{\sigma, x}L^2_y$ and $\mathcal{G}^m_{\sigma, x}H^1_{0,y}$, respectively, then they are also in $\mathcal{G}^m_{\sigma/4, x}L^2_y$ and $\mathcal{G}^m_{\sigma/8, x}H^1_{0,y}$.

\subsection{Novelty and implications}\label{sec:novelty-and-implication}
Let us  highlight the novelties and consequences of Theorem \ref{main-thm} and the discussion in this work. 
Indeed, the improved Gevrey-3 well-posedness of Theorem \ref{main-thm} of System \eqref{LHP} is far from being trivial. 
Standard considerations of the (linearized as well as non-linear) hyperbolic Prandtl system yield to a well-posedness within Gevrey-class 2, at best 
(see e.g.~\cite{LiXu2021}). We show that a suitable cancellation mechanism is inherent to \eqref{LHP}, which is similar to the one presented in \cite{MR3925144}. Contrary to the classical Prandtl equations, however, the hyperbolic behaviour of \eqref{LHP} unlocks more refined estimates, that endow the mentioned improvements of Gevrey-class 3. 

\smallskip
\noindent
In  Section 2 of  \cite{MR3925144}, Dietert and Gerard-Varet provided a rather clear intuition on why the  well-posedness result of Gevrey 2 holds true for the linearised Prandtl equations (later on, their result further address the non-linear system). 
In order to successfully simplify the comprehension of their idea, they made use of calculations involving the Laplace transform on the time variable. 
Then, potentially, some algebraic calculations relating the Laplace variable (in time) with the Fourier variable (in space, along the horizontal direction) allowed to deduce the correct regularity of solutions, since they showed a possible behaviour of the associated semigroup on the linearised system. 
The downside of this approach consists, however, in the fact that the argument seems to lack some final implementation. The authors indeed derived certain a-priori estimates in the Laplace variable, however the inverse Laplace transform does not commute with norms\footnote{Of course, the actual argument in \cite{MR3925144} for the nonlinear system is consistent and rigorous. The aim of the authors in Section 2 was to provide a clear understanding.}, thus this estimates could not be transferred formally to the original solution. 
Our approach differ with the one related to the Laplace transform in \cite{MR3925144}. In particular, \Cref{lemma:Gronwall-improved} provides a simple, yet very useful tool (an improved Gronwall estimate) to infer Gevrey-estimates by energy estimates  (we refer to Section \ref{sec:two} for more details).\\
Furthermore, we give detailed bounds on regularity, life span and explicit quantitative dependence on the shear flow $U_{\mathrm sh}$ (cf.\ \eqref{radii-of-Gevrey-beta-gamma}). For example, if $U''_{\mathrm{sh}}=0$, Theorem \ref{main-thm} shows the \textit{global} well-posedness of \eqref{LHP}, being in correspondence with the results for the classical Prandtl equations with monotonic data (cf.\ \cite{Oleinik1999}). \\
Additionally, we give a detailed discussion on possible improved well-posedness results for the nonlinear hyperbolic Prandtl system \eqref{eq:Prandtl_hyperbolic}. In particular, our work shows that one \textit{cannot} rely on further simplifications of \eqref{eq:Prandtl_hyperbolic} in order to achieve existence results beyond the expected Gevrey 2 class (see, e.g., \cite{LiXu2021} and \cite{PZ2022}). We refer to Section \ref{sec:remarks-on-the-nonlinear-system} for a consideration of three possible nonlinear variants with their advantages and drawbacks in terms of regularity propagation. \\
Finally, a short summary of the remaining parts. The beginning of Section \ref{sec:two} contains an extended overview of the proof of Theorem \ref{main-thm} which is split into six parts. In the aforementioned Section \ref{sec:remarks-on-the-nonlinear-system}, the discussion of possible extensions of the arguments to the nonlinear system is provided.

\section{Proof of Theorem \ref{main-thm}}
\label{sec:two}
We state first the general principles that we set as the basis of our analysis, and we postpone the details of our proof to the remaining paragraphs. 

\noindent 
Our approach is grounded in a similar ansatz as the one developed by Dietert and G\'erard-Varet for the linearised system of the classical Prandtl equations (cf.~Section 2 of \cite{MR3925144}, outline of the strategy). 

\noindent
In \Cref{sec:recasting-uk-into-psik}, we indeed use the Fourier transform along the variable $x\in \mathbb T$, in order to address the behaviour of each mode $u_k:(0,T)\times (0,1)\to \mathbb R$ of the velocity field
\begin{equation*}
    u(t,x) = \sum_{k\in \mathbb Z} u_k(t,y) e^{i kx},
\end{equation*}
at any frequency $k\in \mathbb Z$. Regrettably, the equation of $u_k$ (cf.~\eqref{eq:Prandtl-in-uk-Fourier}) is incapable to derive alone a better stability than analytic  (initial data far more regular then Gevrey 3). A further development is therefore necessary, in order to overcome this first barrier.

\noindent
On this account, following \cite{MR3925144}, we introduce a new state variable $\psi_k:(0,T)\times (0,1)\to \mathbb R$ in \Cref{sec:recasting-uk-into-psik}, which depends on $u_k$ (or rather upon the corresponding stream function $\Phi_k$, $u_k = \partial_y \Phi_k$, cf.~\eqref{eq:psik-not-compact}).
Our main objective is indeed to asses $\psi_k$, in order to remove the (problematic) terms due to $v U_{\rm sh}'$  and $\partial_t v U_{\rm sh}'$. These terms preclude indeed an analysis beyond analytic, thus, by eliminating them, we determine a new form of the main equation (written now in terms of  $\psi_k$, cf.~\eqref{eq:d_ypsik-final-form})
\begin{equation}\label{eq:psik-sketch-of-the-proof}
\begin{aligned}
     \big( (\partial_{t} + 1)(\partial_t  + ik U_{{\rm sh}} )  -\partial_y^2\big)^2
    \partial_y
    \psi_k
    = 
    ik
    [ U_{{\rm sh}}'  , \partial_{y}^2]\big( (\partial_t + 1)\psi_k\big),
\end{aligned}
\end{equation}
which shall eventually facilitate our analysis in Gevrey-class 3.

\noindent
Following \eqref{eq:psik-sketch-of-the-proof}, our approach begins to inherently diverge  with respect to the one of Dietert and {G\'erard-Varet} in \cite{MR3925144}. We avoid entirely their ansatz on the Laplace transform in time, since (despite its clearness) it would lead to the difficulties mentioned in Section \ref{sec:novelty-and-implication}. Contrarily, we develop our analysis around a specific ``weighted'' version of the Gronwall's lemma, which plays somehow the role of cornerstone for our entire proof. Its statement is thus the first that we present in \Cref{sec:weighted-Gronwall} (cf.~\Cref{lemma:Gronwall-improved}).

\noindent 
To be more specific, we take advantage of \eqref{eq:psik-sketch-of-the-proof}, in order to determine a certain meaningful estimate on the  derivatives $\partial_y^2 \psi_k$ and  $(\partial_t+1) \psi_k$, as described in details in \Cref{prop:estimate-of-dtdypsik}. This estimate can be expressed essentially as
\begin{equation}\label{Gronwall-sketch-proof}
    \|\psi_k(t) \| \leq g_k(t) + C(t)|k| \int_0^t(t-s)^2  \|\psi_k(s)\| ds,
\end{equation}
where $t\in (0,T)\to \|\psi_k(t) \|$ represents the $L^2$-norms (in $y\in (0,1)$) of $\partial_y^2 \psi_k$ and  $(\partial_t+1) \psi_k$ (for the complete version, we refer to \Cref{prop:estimate-of-dtdypsik}). 
The function $g_k$ and $C$ in \eqref{Gronwall-sketch-proof} are non-decreasing, while the integral is expressed also in terms of a ``weight'' in time: the kernel $(t-s)^2$. 

\noindent
The kernel $(t-s)^2$ in \eqref{Gronwall-sketch-proof} unlocks the regularity Gevrey-class 3, for the derivatives  $\partial_y^2\psi_k$ and $(\partial_t+1)\psi_k$. 
To grasp this principle, we shall first remark that,
in its absence (thus within a standard Gronwall inequality), we may at best derive an estimate of the form $\|\psi_k(t) \|\leq g_k(0)\exp(tC(t)|k|)$, where the modes $\psi_k$ growth exponentially as $ |k|$ (the setting of analytic solutions). The presence of the kernel provides us however better information: because of the weighted Gronwall inequality in \Cref{lemma:Gronwall-improved}, the norm $\|\psi_k(t) \|$ can at worst growth as $g_k(0)\exp(\sqrt[3]{tC(t)|k|})$, i.e.~exponentially as $|k|^{1/3}$, the framework of Gevrey-class 3. 

\smallskip
\noindent
The remaining sections are devoted to transfer the aforementioned estimate of $\partial_y^2\psi_k$ and $(\partial_t+1)\psi_k$ first to $\psi_k$ (cf.~\Cref{lemma:Gevrey3-psi}) and secondly to $u_k$ (cf~\Cref{prop:transferring-Gevrey-to-uk}).  \Cref{prop:transferring-Gevrey-to-uk} is moreover essential to determine the final solution $u$ of \eqref{LHP}, which is Gevrey-class 3 in $x\in \mathbb T$. Furthermore, this result provides the final estimate \eqref{est:main-thm-inequality-Gevrey} on the Gevrey-norm of $u$ at any time $t\in [0, T_\sigma)$, with also the corresponding constant $D_\sigma(U_{\rm sh})$ in \eqref{def:Dsigma}.

\noindent 
Summarising, the forthcoming sections are structured as follows:
\begin{itemize}
    \item \Cref{sec:weighted-Gronwall} and \Cref{lemma:Gronwall-improved} are devoted to the proof of the ``weighted'' Gronwall's inequality.
    \item In \Cref{sec:recasting-uk-into-psik} we introduce the new state variable $\psi_k$ and derive the corresponding equation \eqref{eq:psik-sketch-of-the-proof}.
    \item In \Cref{sec:reaching-Gevrey-3}, we first state the main inequality \eqref{Gronwall-sketch-proof} in \Cref{prop:estimate-of-dtdypsik} (whose proof is postponed to \Cref{sec:test-function} and \Cref{sec:the-main-estimate}). We furthermore transfer the Gevrey estimates of $\partial_y^2\psi_k$ and $(\partial_t +1)\psi_k $ to 
    $\psi_k$ in \Cref{lemma:Gevrey3-psi}.
    
    \item In \Cref{sec:transferring-G3-to-u}, with \Cref{prop:transferring-Gevrey-to-uk}, we transfer the Gevrey estimates to $u_k$ and build our final solution $u$ of System \eqref{LHP}. To conclude the proof of the main \Cref{main-thm}, we  determine moreover the estimate \eqref{est:main-thm-inequality-Gevrey} on the Gevrey norm of the solution.
    
    \item Finally, \Cref{sec:test-function} and \Cref{sec:the-main-estimate} are devoted to the proof of \Cref{prop:estimate-of-dtdypsik} and the main inequality \eqref{Gronwall-sketch-proof}.
\end{itemize}

\subsection{A weighted Gronwall inequality}\label{sec:weighted-Gronwall}

One of the main ingredient used to prove Theorem \ref{main-thm} is the following Gronwall-type Lemma. It asserts that  any non-negative function, which satisfies a ``weighted'' Gronwall's inequality proportional to a suitable time-dependent function $\lambda(t)^3$, can not grow up exponentially faster than $\lambda(t)t$.
\begin{lemma}\label{lemma:Gronwall-improved}
    Let $T>0$ and $f:[0, T)\to [0, \infty)$ be a non-negative continuous function, satisfying
    \begin{equation}\label{ineq1-gronwall-type-lemma}
        f(t) \leq g(t) + 
        \frac{\lambda(t)^3}{2} 
        \int_0^t (t-s)^2 f(s) ds,
    \end{equation}
    for two continuous functions $\lambda, g:[0, T)\to [0, \infty)$, that are non-negative and non-decreasing. Then the following inequality holds true at any time $t\in [0,T)$:
    \begin{equation}\label{main-est-gronwall-lemma}
        f(t) \leq g(t) \Big(1+\frac{(\lambda(t) t)^3}{6}\Big)e^{\lambda(t) t }.
    \end{equation}
\end{lemma}
\begin{proof}
    We begin with by defining $\omega$ in $\mathcal{C}^3([0, T))$ as the following integral: \eqref{ineq1-gronwall-type-lemma} as
    \begin{align*}
        \omega(t):=\frac{1}{2}
        \int_0^t (t-s)^2 f(s) ds = 
        \int_0^t  f(s) 
        \int_{s}^t \int_\tau^t dr  d\tau ds
        = 
        \int_0^t 
        \int_{0}^r \int_0^\tau  f(s) ds d\tau d r.
    \end{align*}
    The function $\omega$ is everywhere non-negative in $[0, T)$. Furthermore, at $t = 0$, $\omega(0)$ and its derivatives $\omega'(0),\omega''(0)$ are all identically null. 
    We next write inequality \eqref{ineq1-gronwall-type-lemma} in terms of $\omega$:
    \begin{equation}\label{omega'''-lemma-Gronwall}
        \omega'''(t) \leq g(t) +  \lambda(t)^3 \omega(t), \quad \text{for all}
        \quad t\in [0, T).
    \end{equation}
    Hence, we fix a general time $\tilde t\in (0,T)$ and we momentarily consider only values of $t$ within $[0,\tilde t]$. We multiply equation \eqref{omega'''-lemma-Gronwall} with $e^{-\lambda(\tilde t)t}$ (where $\lambda(\tilde t)$ is fixed and  plays momentarily the role of a constant). By means of standard calculations on the derivatives, we gather that
    \begin{equation}\label{est:derivatives-of-omega}
    \begin{aligned}
        \omega'''(t) e^{-\lambda(\tilde t)t}
        =
        \frac{d^3}{dt^3}
        \Big(\omega(t) e^{-\lambda(\tilde t)t}\Big) 
        &
        \!+\! 
        3
        \frac{d^2}{dt^2}
        \Big(\lambda(\tilde t) (\omega(t)
        e^{-\lambda(\tilde t) t}
        \Big)
        \!+\! 
        3 
        \frac{d}{dt}
        \Big(\lambda(\tilde t)^2\omega(t) e^{-\lambda(\tilde t) t}
        \Big)
        \!+\!
        \lambda(\tilde t)^3\omega(t)e^{-\lambda(\tilde t) t}
        \\
        &\leq \big( g(t)+ \lambda(t)^3 \omega(t)\big)e^{-\lambda(\tilde t) t}
        \leq   g( t) e^{-\lambda(\tilde t) t}+ \lambda(\tilde t)^3 \omega( t)e^{-\lambda(\tilde t) t}.
    \end{aligned}
    \end{equation}
    In the last inequality, we have used that $\lambda$ is non-decreasing and non-negative, $\omega \geq 0 $ and that $t< \tilde t$. We shall now remark that 
    the term $\lambda(\tilde t)^3\omega(t)e^{-\lambda(\tilde t) t}$ cancel out and the left-hand side of \eqref{est:derivatives-of-omega} is hence left with only time derivatives. We are hence in the condition to integrate \eqref{est:derivatives-of-omega} along $[0,t]$, to gather that 
    \begin{equation*}
        \frac{d^2}{dt^2}
        \Big(
        \omega(t) e^{-\lambda(\tilde t)t}
        \Big) + 
        3\lambda(\tilde t) 
        \frac{d}{dt}
        \Big(\omega(t)e^{-\lambda(\tilde t) t}
        \Big)+ 3 \lambda(\tilde t)^2 \omega(t) e^{-\lambda(\tilde t) t}
       \leq \int_0^t g(s)e^{-\lambda(\tilde t) s} ds
       \leq 
       \int_0^t g(s)  ds.
    \end{equation*}
    We can drop the term $3 \lambda(\tilde t)^2 \omega(t) e^{-\lambda(\tilde t) t}$ at the left-hand side (since it is positive) and integrate once more along the interval $[0,t]$, for a general $t\in [0, \tilde t]$:
    \begin{equation*}
        \frac{d}{dt}
        \Big(
            \omega(t) e^{-\lambda(\tilde t)t}
        \Big) + 3\lambda(\tilde t) \omega(t)e^{-\lambda(\tilde t)t}
       \leq \int_0^t\int_0^s g(z) dz = 
       \int_0^t(t-z) g(z) dz.
    \end{equation*}
    Since both $\lambda(\tilde t)$ and $\omega(t)$ are positive, we can  drop  $3\lambda(\tilde t) \omega(t)e^{-\lambda(\tilde t)t}$ and integrate a final time along $(0,t)$:
    \begin{equation*}
     \omega(t) e^{-\lambda(\tilde t)t}
     \leq 
     \int_0^t\int_0^s(s-z) g(z) dzds = 
     \int_0^t g(z)\frac{(t-z)^2}{2} dz
     \quad \Rightarrow \quad 
     \omega(t)\leq 
     e^{\lambda(\tilde t)t}
     \int_0^t g(z)\frac{(t-z)^2}{2} dz.
   \end{equation*}
   We are now in the condition to combine the last relation in $t=\tilde t$ together with \eqref{ineq1-gronwall-type-lemma}, which ensures that
   \begin{equation*}
   \begin{aligned}
       f(\tilde t) 
       &\leq 
       g(\tilde t) + 
       e^{\lambda(\tilde t)\tilde t}
       \lambda(\tilde t)^3
       \int_0^{\tilde t} g(s)\frac{(\tilde t-s)^2}{2} ds\\
       &\leq 
       g(\tilde t) + 
       e^{\lambda(\tilde t)\tilde t}
       \lambda(\tilde t)^3
       g(\tilde t) 
       \int_0^{\tilde t}\frac{(\tilde t-s)^2}{2} ds 
       \leq  
       g(\tilde t) \Big( 1+ \frac{ \lambda(\tilde t)^3\tilde t^3}{6}\Big)e^{\lambda(\tilde t)\tilde t}.
   \end{aligned}
   \end{equation*}
   Re-denoting $\tilde t = t$ and from its arbitrariness in $(0,T)$, we finally achieve inequality \eqref{main-est-gronwall-lemma}
   (also remarking that \eqref{main-est-gronwall-lemma} is trivially satisfied in $t = 0$). 
   This concludes the proof of the lemma.
\end{proof}

\subsection{The stream function formulation}\label{sec:recasting-uk-into-psik}

In this section, we begin developing our analysis of System \eqref{LHP} and we first decompose the corresponding equations in terms of several Fourier coefficients $u_k: (t,y)\in  (0, T_\sigma) \times (0,1) \to \mathbb R$ of the velocity field $u$, at any frequency $k\in \mathbb Z$. The lifespan $T_\sigma>0$ (denoted by $T$ in \eqref{LHP}) shall be considered from now on as in \eqref{def:Tsigma} of \Cref{main-thm}, nevertheless its form will play a major role only starting from  \Cref{sec:reaching-Gevrey-3}.

\noindent
Eventually, we will build the final solution $(u,v)$ of \eqref{LHP}, 
by invoking the Fourier Series with respect to the variable $x\in \mathbb{T}$ and the divergence-free condition $\partial_x u + \partial_y v = 0$ (which at any frequency is $ik u_k + \partial_y v_k = 0$)
\begin{equation}\label{Fourier-series}
\begin{alignedat}{4}
    u(t,x,y) &= \sum_{k\in \mathbb{ Z}} u_k(t,y) e^{ik x},\quad
    u_k(t,y) &&:= \frac{1}{2\pi} \int_\mathbb{T} u(t,x,y) e^{-i x k}dx, \\
    v(t,x,y) &= \sum_{k\in \mathbb{ Z}} v_k(t,y) e^{ik x},\quad
    v_k(t,y) &&:=-ik \int_0^yu_k(t,z)dz,
\end{alignedat}
\end{equation}
however we shall first determine some uniform estimates on $(u_k)_{k\in \mathbb Z}$, in order to provide a sense of the series above. 
Hence, we begin with by considering System \eqref{LHP} rather as a family of PDEs in the  variables $(t,y)\in (0,T_\sigma)\times (0,1)$, which depend upon each frequency $k\in \mathbb{Z}$:
\begin{equation}\label{eq:Prandtl-in-uk-Fourier}
    \system{
    \begin{alignedat}{4}
    & \partial_{t}^2 u_k + ik U_{{\rm sh}} \partial_t u_k + U_{{\rm sh}}'  \partial_t v_k +
    \partial_tu_k  + ik U_{{\rm sh}}  u_k + v_k  U'_{{\rm sh}} - \partial_{y}^2 u_k =0
    \qquad &&(0,T_\sigma) \times (0,1),\\
    & ik u_k + \partial_y v_k =0
    &&(0,T_\sigma) \times (0,1),
    \\
    & \left. \pare{u_k, \partial_t u_{k}}\right|_{t=0} =\pare{u_{{\rm in}, k}, u_{t, {\rm in}, k}}
    &&\hspace{1.4cm}(0,1),
    \\
    & \left. \pare{u_k,v_k}\right|_{y=0,1}=(0,0)
     &&(0,T_\sigma).
    \end{alignedat}
    }
\end{equation}
The second equation $ik u_k + \partial_y v_k =0$ and the boundary conditions $v_{k|y = 0}= 0$ allow to interpret System \eqref{eq:Prandtl-in-uk-Fourier} only on the state variable $u_k$, since the vertical component $v_k$ is explicitly determined by \eqref{Fourier-series}. The initial data $u_{{\rm in}, k}$ and $u_{t, {\rm in}, k}$ are in $H^1_0(0,1)$ and $L^2(0,1)$, respectively, since $u_{\rm in}$ and $u_{t,{\rm in}}$ are in $ \mathcal{G}^3_{\sigma, x}H^1_{0,y}$ and $ \mathcal{G}^3_{\sigma, x}L^2_{0,y}$, as described by \Cref{def:Gevrey-regularity}.

\noindent 
We hence state the following result about the existence and uniqueness of solutions for System \eqref{eq:Prandtl-in-uk-Fourier}.
\begin{prop}\label{prop:existence-uniqueness-uk}
 For any fixed frequency $k\in \mathbb Z$ and any initial data $\pare{u_{{\rm in}, k}, u_{t, {\rm in}, k}}$  in $H^1_0(0,1)\times L^2(0,1)$ there exists a unique solution $u_k:[0,T)\times (0,1)\to \mathbb R$ of \eqref{eq:Prandtl-in-uk-Fourier}, which belongs to
 \begin{equation}\label{where-uk-belongs}
     (u_k, \partial_y u_k)\in  
     \mathcal{C}([0, T], H^1_0 ),\qquad 
     \partial_t u_k \in L^2(0,T;L^2),
 \end{equation}
 for any real time $T>0$.
\end{prop}
\noindent 
We shall here remark that the real $T>0$ may not correpond to $T_\sigma$, since the lifespan in \eqref{def:Tsigma} may be $T_\sigma =+\infty$ (for shear flow with $U_{\rm sh}''\equiv U_{\rm sh}''' \equiv 0$). In this case, we shall always treat $u_k$ as in \eqref{where-uk-belongs}, for any $0< T<T_\sigma = +\infty$.
Contrarily, if $T_\sigma<+\infty$ (which is satisfied for more general $U_{\rm sh}$), then we replace $T$ in \eqref{where-uk-belongs} directly with $T_\sigma$.
\begin{proof}
For the sake of simplicity, we provide here only a sketch, since the result can be shown through standard arguments on linear PDEs. We remark indeed that \Cref{eq:Prandtl-in-uk-Fourier} can be written as a 1D damped wave equation with Dirichelet boundary conditions
\begin{align}
\label{eq:1Dwaveq}
    \left( \Box + \partial_t \right) u_k = F_k, 
    \qquad 
    \Box = \partial_t^2 - \partial_y^2,\qquad 
    \left. u_k \right|_{y=0,1}=0 , 
\end{align}
and forcing term $F_k$, which depends linearly on $u_k$ and is given by 
\begin{equation*}
    F_k = - 
    \Big(ik U_{{\rm sh}} \partial_t u_k -ik U_{{\rm sh}}'  \partial_t \int_0^y  u_k(t,z) dz 
     + ik U_{{\rm sh}}  u_k -ik  U'_{{\rm sh}} \int_0^y  u_k(t,z) dz \Big). 
\end{equation*}
Making use of a standard computation combined with the Poincar\'e inequality, we infer that $F_k$ satisfies
\begin{equation}
    \left\| F_k\right\|_{H^1_0}\leq C|k| \left\|U_{{\rm sh}}\right\|_{W^{3, \infty}}\left(  \left\| \partial_t u_k \right\|_{H^1_0} + \left\| \partial_y u_k \right\|_{H^1_0} \right) , 
\end{equation}
for a suitable positive constant $C>0$. We can hence apply a standard Galerkin method to deduce the existence of a unique solution within the function space of \eqref{where-uk-belongs}.
\end{proof}
\noindent
We come back now to System \eqref{eq:Prandtl-in-uk-Fourier}. Since $u_k(t,\cdot)$ is divergent free, it can be written in terms of a stream function $\Phi_k = \Phi_k(t,y)$, which is in $\mathcal{C}([0, T], H^2)$ with $\partial_t \Phi_k \in L^2(0,T;H^1(0,1))$, for any real $T\leq T_\sigma$. Furthermore, because $ik \Phi_k =v_k$, the function $\Phi_k$ is identically null in $y = 0$ (in the sense of trace), therefore
\begin{equation}\label{def-Phik}
    u_k(t,y) = \partial_y \Phi_k(t,y)\quad
    \Big(\text{i.e.~}\Phi_k(t,y) := \int_0^y  u_k(t,z)dz
    \Big)
    \quad 
    \Rightarrow\quad 
    v_k(t,y) =- ik\Phi_k(t,y).
\end{equation}
The initial data of $\Phi_k$ at any $ y\in (0,1)$ are given by
\begin{equation*}
    \Phi_{{\rm in}, k}(y) := \int_0^y u_{{\rm in}, k}(z)dz,\quad 
    \Phi_{t,{\rm in}, k}(y) := \int_0^y u_{t,{\rm in}, k}(z)dz,
\end{equation*}
which ensures $\Phi_{{\rm in}, k}\in H^2(0,1)\cap H^1_0(0,1)$ and $\Phi_{t,{\rm in}, k}\in H^1(0,1)$(remark that $\Phi_{{\rm in}, k}(1) = 0$, since $u_{{\rm in}, k}$ is average free). 

\noindent
From \eqref{eq:Prandtl-in-uk-Fourier}, we deduce that $\Phi_k$ is solution in $(0,T_\sigma)\times (0,1)$ of the following system:
\begin{equation}\label{eq:Phik-not-compact}
    \system{
    \begin{alignedat}{4}
    & \partial_{t}^2 \partial_y \Phi_k  \!+\! ik U_{{\rm sh}}(y)\partial_t \partial_y \Phi_k  \!-\! ik U_{{\rm sh}}'(y) \partial_t \Phi_k 
    \!+\!
    \partial_t\partial_y \Phi_k \! +\! U_{{\rm sh}}(y) ik \partial_y \Phi_k \!-\! ik \Phi_k  U'_{{\rm sh}}(y) - \partial_{y}^3 \Phi_k =0,
    \\
    & \left. \pare{\Phi_k , \partial_t \Phi_{k}}\right|_{t=0} =\pare{\Phi_{{\rm in}, k}, \Phi_{t,{\rm in}, k}}
    \\
    & \left. \pare{\partial_y \Phi_k} \right|_{y=0,1}=0,\qquad 
    \Phi_{k|y = 0} = 0.
    \end{alignedat}
    }
\end{equation}
We next provide some heuristics on the $L^2$-estimates satisfied by $\Phi_k$ and show that, unfortunately, System \eqref{eq:Phik-not-compact}  (as it is written) may not prevent the stream function $\Phi_k$ to exponentially growth as $\exp(a|k|)$, for some positive $a>0$. 
To this end, we first isolate the linear operator in equation \eqref{eq:Phik-not-compact} that affects only the derivative $\partial_y \Phi_k$ and transfer the remaining terms in $\Phi_k$ on the right-hand side of the identity:
\begin{equation}\label{eq:Phik-compact}
    ( (\partial_{t} + 1)(\partial_t  + ik U_{{\rm sh}} )  -\partial_y^2) \partial_y \Phi_k 
    =
    (\partial_t + 1)ik U_{{\rm sh}}' \Phi_k.
\end{equation}
We will shortly see that the operator on the left-hand side of \eqref{eq:Phik-compact} is crucial for our next analysis  (in particular, to define a new state variable $\psi_k$ in \eqref{eq:psik-not-compact}). We first outline, however, that, in the current form, equation \eqref{eq:Phik-compact} is still ineffective and does not predict the crucial Gevrey-3 regularity of our solutions. Indeed, we infer that a standard energy approach would provide (at best) an $L^2$-estimate of $\partial_y\Phi_k=u_k$ of the form 
\begin{equation}\label{est:just-analytic}
    \frac{1}{2}
    \frac{d}{dt}
    \Big(
        \| (\partial_t  +1 ) \partial_y \Phi_k \|_{L^2}^2 + \| \partial_t \partial_y \Phi_k \|_{L^2}^2 + 
        2\|  \partial_y^2 \Phi_k \|_{L^2}^2
    \Big)
    \leq 
    C|k|
    \Big(
        \| (\partial_t  +1 ) \partial_y \Phi_k \|_{L^2}^2 + \| \partial_t \partial_y \Phi_k \|_{L^2}^2
    \Big),
\end{equation}
for a suitable positive constant $C$, which is also calibrated with the following Poincar\'e-type inequality of $\Phi_k$ in the domain $y\in (0,1)$: 
\begin{equation}
\begin{aligned}
    \| \Phi_k(t) \|_{L^2} 
    &= 
    \bigg(
        \int_0^1 | \Phi_k(t,y)|^2dy
    \bigg)^\frac{1}{2}
    =
    \bigg(
        \int_0^1 \Big| \int_0^y\partial_y\Phi_k(t,z)dz\Big|^2dy
    \bigg)^\frac{1}{2}\\
    &\leq
     \bigg(
        \int_0^1 y \int_0^y\Big|\partial_y\Phi_k(t,z)\Big|^2dzdy
    \bigg)^\frac{1}{2}
    \leq
    \frac{1}{2}
     \bigg(
        \int_0^1 \Big|\partial_y\Phi_k(t,z)\Big|^2dz
    \bigg)^\frac{1}{2}
    \leq  \| \partial_y \Phi_k(t) \|_{L^2}.
\end{aligned}
\end{equation}
Hence, roughly speaking, in this regime the $L^2$-norms of 
$\partial_y\Phi_k$ and $\partial_t\partial_y\Phi_k$  would growth exponentially as 
$e^{c|k|}(\| \partial_y \Phi_{\rm in, k} \|_{L^2}+\| \partial_t \partial_y \Phi_{\rm in, k} \|_{L^2})  $, a setting which is typical of analytic solutions (which are of course much more regular than any Gevrey-class $m$, $m>1$).
To achieve the Gevrey regularity, we shall therefore perform a further development. To this end, we introduce a new state variable $\psi_k:[0,T_\sigma)\times (0,1)\to \mathbb R$.

\noindent
Following the approach used in \cite{MR3925144} for the classical Prandtl equation,  $\psi_k = \psi_k(t,y)$ is chosen in a form that gets rid of the terms 
$ik U_{\rm sh}'\partial_t \Phi_k$  and $ik U_{\rm sh}'  \Phi_k$ at the left-hand side of \eqref{eq:Phik-not-compact}.
More precisely, we define $\psi_k$ as the unique solution in  $L^\infty(0,T;H^2)\cap L^2 (0,T;H^2)$ of the following PDE:
\begin{equation}\label{eq:psik-not-compact}
    \system{
    \begin{alignedat}{4}
    &( (\partial_{t} + 1)(\partial_t  + ik U_{{\rm sh}} )  -\partial_y^2) \psi_k 
    = \Phi_k 
    \qquad 
    &&(0,T) \times (0,1),
    \\
    & \left. \pare{\psi_k , \partial_t \psi_{k}}\right|_{t=0} =\pare{0, 0}
    &&\hspace{1.4cm}(0,1),
    \\
    & \left. \psi_k \right|_{y=0,1}=0
    &&(0,T).
    \end{alignedat}
    }
\end{equation}
The most compelling reason for this definition is a meaningful cancellation that occurs when coupling \eqref{eq:psik-not-compact} together with \eqref{eq:Phik-compact}. More precisely, equation \eqref{eq:Phik-compact} implies that $\psi_k$ satisfies
\begin{equation*}
    ( (\partial_{t} + 1)(\partial_t  + ik U_{{\rm sh}} )  -\partial_y^2) \partial_y 
     ( (\partial_{t} + 1)(\partial_t  + ik U_{{\rm sh}} )  -\partial_y^2)\psi_k 
    =
    ik U_{{\rm sh}}' (\partial_t + 1)\Phi_k, 
\end{equation*}
which is 
\begin{equation}     \label{firstCommutator}
    \big( (\partial_{t} + 1)(\partial_t  + ik U_{{\rm sh}} )  -\partial_y^2\big)^2
    \partial_y
    \psi_k
    +
    \big( (\partial_{t} + 1)(\partial_t  + i k U_{{\rm sh}} )  -\partial_y^2\big)
    \big(
    ikU_{{\rm sh}}' 
    (\partial_t + 1)\psi_k
    \big)
    = ik U_{{\rm sh}}' (\partial_t + 1)\Phi_k.
\end{equation}
The second term at left-hand side of \eqref{firstCommutator} almost coincides with $ik U_{{\rm sh}}' (\partial_t + 1)\Phi_k$ at the right-hand side. 
To complete the aforementioned cancellation, we first invoke the commutator $ik[ U_{{\rm sh}}' , \partial_{y}^2]\psi_k = ik U_{{\rm sh}}'\partial_y^2(\psi_k)- ik\partial_y^2( U_{{\rm sh}}'\psi_k)$, in order to write the second term in \eqref{firstCommutator} as
\begin{equation*}
    \big( (\partial_{t} + 1)(\partial_t  + i k U_{{\rm sh}} )  -\partial_y^2\big)
    \big(
    ikU_{{\rm sh}}' 
    (\partial_t + 1)\psi_k
    \big)
    = 
    ikU_{{\rm sh}}' 
    \big( (\partial_{t} + 1)(\partial_t  + i k U_{{\rm sh}} )  -\partial_y^2\big)
    \big(
    (\partial_t + 1)\psi_k
    \big)
    -
    ik[ U_{{\rm sh}}' , \partial_{y}^2]\psi_k.
\end{equation*}
Finally, we plug this identity into \eqref{firstCommutator}, to gather
\begin{equation}\label{I-have-no-more-names}
\begin{aligned}
    \big( (\partial_{t} + 1)(\partial_t  + ik U_{{\rm sh}} )  -\partial_y^2\big)^2
    \partial_y
    \psi_k
    +
    ikU_{{\rm sh}}'
    (\partial_t +1)
    \Big(
    \underbrace{
        \big( (\partial_{t} + 1)(\partial_t  + ik U_{{\rm sh}} )  -\partial_y^2\big)
    \psi_k
    }_{= \Phi_k}
    \Big) + \\
    - [ik U_{{\rm sh}}'(y) , \partial_{y}^2]\big( (\partial_t + 1)\psi_k\big)
    = ik U_{{\rm sh}}' (\partial_t + 1)\Phi_k.
\end{aligned}
\end{equation}
Recalling that $\psi_k$ satisfies \eqref{eq:psik-not-compact}, we remark that $ik U_{{\rm sh}}' (\partial_t + 1)\Phi_k$ appears both on the left- and right-hand sides of \eqref{I-have-no-more-names}. We thus obtain the following final form of the $\psi_k$-equation:
\begin{equation}\label{eq:d_ypsik-final-form}
\begin{aligned}
     \big( (\partial_{t} + 1)(\partial_t  + ik U_{{\rm sh}} )  -\partial_y^2\big)^2
    \partial_y
    \psi_k
    = 
    ik
   [ U_{{\rm sh}}'  , \partial_{y}^2]\big( (\partial_t + 1)\psi_k\big),
\end{aligned}
\end{equation}
which was indeed claimed at the beginning in \eqref{eq:psik-sketch-of-the-proof}. We shall remark that \eqref{eq:d_ypsik-final-form} still presents a forcing term $ ik
   [ U_{{\rm sh}}'  , \partial_{y}^2]\big( (\partial_t + 1)\psi_k\big)$, which growths linearly like $|k|$ at high frequencies $|k|\gg 1$ (similarly as $ikU_{\rm sh}'(\partial_t+1)\Phi_k$ in \eqref{eq:Phik-compact}). Nonetheless,  the operator $  ( (\partial_{t} + 1)(\partial_t  + ik U_{{\rm sh}} )  -\partial_y^2 )^2$ on $\partial_y
    \psi_k$ has now doubled in order (in comparison with just $  ( (\partial_{t} + 1)(\partial_t  + ik U_{{\rm sh}} )  -\partial_y^2 )$ in \eqref{eq:Phik-compact}). This will unlock more-refined estimates on $\partial_y \psi_k$ (and its derivatives) than the ones in  \eqref{est:just-analytic} for $\partial_y\Phi_k$. 
    We formalise these heuristics in the next sections.

\subsection{Reaching Gevrey-class 3}\label{sec:reaching-Gevrey-3}
In the forthcoming analysis, we illustrate how the derived equation \eqref{eq:d_ypsik-final-form} succeeds in enabling an $L^2$-estimate of Gevrey-3 type to the new state variable $\partial_y \psi_k$, as well its time derivative $\partial_t \partial_y \psi_k$. This estimate is a consequence of the following statement, that we set as the basis of our development. It guarantees that the functions $(\partial_t +1)\partial_y \psi_k$ and $\partial_y^2 \psi_k$ satisfy an improved Gronwall-type inequality, as described by Lemma \ref{lemma:Gronwall-improved}.

\begin{prop}\label{prop:estimate-of-dtdypsik}
    The following estimate on the functions $(\partial_t + 1)\partial_y \psi_k$ and $\partial_y^2 \psi_k$ holds true, for any frequency $k\in \mathbb{Z}$ and at any time $t \in (0,T_\sigma)$:
    \begin{equation}\label{main-est:prop-dtdypsik}
    \begin{aligned}
        \sup_{s\in [0,t]}
        \Big\{
            \big\| (\partial_t &+1)\partial_y \psi_k(s) \big\|_{L^2}+
            \big\|  \partial_y^2 \psi_k(s) \big\|_{L^2}
        \Big\}\\
        &\leq 
        g_k(t)
        +
        \frac{\lambda_k(t)^3}{2}
        \int_0^t 
        (t-s)^2
        \sup_{\tau \in [0,s]}
         \Big\{
            \big\| (\partial_t +1)\partial_y \psi_k(\tau) \big\|_{L^2}+
            \big\|  \partial_y^2 \psi_k(\tau) \big\|_{L^2}
        \Big\}ds.
    \end{aligned}
    \end{equation}
    The functions $g_k$, $\lambda_k$ are increasing in time and depend uniquely on $k\in \mathbb Z$, the shear flow $U_{\rm sh}$ and the initial data $(u_{{\rm in}, k}, u_{t,{\rm in}, k}, \Phi_{{\rm in}, k})$. More precisely, $g_k$, $\lambda_k$ are defined by 
    \begin{equation}\label{def:gk-lambdak}
    \begin{aligned}
        g_k(t) 
        &:= 
        4t 
        \Big\{
            |k|
            \Big(
                \|U_{\rm sh}' \|_{L^\infty}\| \Phi_{\rm in, k}  \|_{L^2} + 
                \|U_{\rm sh}  \|_{L^\infty}\| u_{\rm in, k}     \|_{L^2}
            \Big)
            +
            \| u_{t,\rm in, k}     \|_{L^2}
            +
            (3+\sqrt{2})
            \| u_{\rm in, k}     \|_{L^2}
        \Big\}
        \\
        \lambda_k(t)
        &:= 2^{\frac{5}{6}}|k|^\frac{1}{3}
        \Big(
                \| U_{\rm sh}'''\|_{L^\infty}
                +
                2
                \| U_{\rm sh}'' \|_{L^\infty}
        \Big)^\frac{1}{3}
        e^{\frac{t}{3}}.
    \end{aligned}
    \end{equation}
    for $k\in \mathbb{Z}$ and $t \in (0,T_\sigma)$.
\end{prop}
\noindent 
Since the proof of this Proposition is rather technical, we postpone it to Section \ref{sec:test-function} and we focus the next paragraphs on the remaining steps to prove Theorem \ref{main-thm}.

\smallskip
\noindent
Let us observe that inequality \eqref{main-est:prop-dtdypsik} encompasses the form \eqref{ineq1-gronwall-type-lemma} of Lemma \ref{lemma:Gronwall-improved}. This ensures therefore that 
the following improved Gronwall's inequality holds true for any $t\in [0,T_\sigma)$:
\begin{align*}
    \big\| (\partial_t +1)\partial_y \psi_k(t) \big\|_{L^2}
    +
    \big\| \partial_y^2 \psi_k(t) \big\|_{L^2}
    \leq 
    \sup_{s\in [0,t]}
    &\big\| (\partial_t +1)\partial_y \psi_k(s) \big\|_{L^2}
    \leq 
    g_k(t) 
    \left(
        1+ 
        \frac{(\lambda_k(t)t)^3}{6}
    \right)
    e^{\lambda_k(t) t}.
\end{align*}
We have essentially achieved the claimed regularity of Gevrey $3$, at least for $(\partial_t +1)\partial_y \psi_k$ and $\partial_y^2\psi_k$, since the definitions of $g_k$ and $\lambda_k$ in \eqref{def:gk-lambdak} imply that
\begin{equation}\label{est:(dt+1)dypsik-first-est-with-Gevrey-3}
\begin{aligned}
    \sup_{s\in [0,t]}
    \Big\{
        \big\| (\partial_t \!+\!1)\partial_y \psi_k(s) \big\|_{L^2}
        \!+\!
        \big\| \partial_y^2 \psi_k(s) \big\|_{L^2}\!
    \Big\}
    \!\leq\! 
    4t 
        \bigg\{\!
            |k|
            \Big(
                \|U_{\rm sh}' \|_{L^\infty}\| \Phi_{\rm in, k}  \|_{L^2} + 
                \|U_{\rm sh}  \|_{L^\infty}\| u_{\rm in, k}     \|_{L^2}
            \Big)
            \!+\!
            \| u_{t,\rm in, k}     \|_{L^2}
            \!+
            \\
            +(3\!+\!\sqrt{2})
            \| u_{\rm in, k}     \|_{L^2}\!
        \bigg\}
        \bigg(
        \!1\!+\! 
        \frac{2\sqrt{2}|k|\Big(
        \| U_{{\rm sh}}''' \|_{L^\infty}
        +
        2\| U_{{\rm sh}}'' \|_{L^\infty}
        \Big)e^t t^3}{3}
        \bigg)
        \exp
        \bigg\{
        |k|^\frac{1}{3}2^\frac{5}{6}\Big(
        \| U_{{\rm sh}}''' \|_{L^\infty}
        \!+\!
        2\| U_{{\rm sh}}'' \|_{L^\infty}
        \Big)^\frac{1}{3}e^\frac{t}{3}t
        \bigg\}.
\end{aligned}
\end{equation}
In particular, the $L^2$-norm increases exponentially at worst as  $|k|^{1/3}$, which we can counteract (at least locally in time), by imposing that the initial data exponentially decay with the same order. To formalise this principle, we shall however first transfer all frequencies $|k|$ of \eqref{est:(dt+1)dypsik-first-est-with-Gevrey-3} to the exponential function, as well as transfer these Gevrey-3 estimates also to $\psi_k, \partial_y \psi_k$ and $\partial_t \partial_y \psi_k$ (instead of just $(\partial_t+1) \partial_y \psi_k$ and $ \partial_y^2 \psi_k$). 
We cope with this issue in the following lemma.
\begin{lemma}\label{lemma:Gevrey3-psi}
    Assume that the sequence of initial data $(u_{{\rm in}, k}, u_{t,{\rm in}, k}, \Phi_{{\rm in}, k})_{k\in \mathbb Z}$ satisfies 
    \begin{equation}
        \sup_{k\in \mathbb Z}
        \Big\{
        e^{\sigma|k|^\frac{1}{3}}
        \Big(
        \|  u_{{\rm in}, k}       \|_{L^2}
        +
        \| u_{t,{\rm in}, k} \|_{L^2}
        +
        \| \Phi_{{\rm in}, k} \|_{L^2}
        \Big)
    \Big\}<+\infty,
    \end{equation}
    for a given radius $\sigma>0$. Let $\alpha:[0,T_\sigma)\to \mathbb R_+$ be defined by
    \begin{equation}\label{def:alpha-radius}
        \alpha(t) := \frac{\sigma}{2}
         -
        2^\frac{5}{6}
        \Big(
        \| U_{{\rm sh}}''' \|_{L^\infty}
        +
        2\| U_{{\rm sh}}'' \|_{L^\infty}
        \Big)^\frac{1}{3}e^\frac{t}{3}t\geq 0.
    \end{equation}
    Then the sequence $(\psi_k)_{k\in \mathbb N}$ generates a function $\psi: [0,T_\sigma)\times \mathbb T\times (0,1)\to \mathbb R$ through the Fourier series
    \begin{equation}\label{def-psi-through-Fourier-Series}
        \psi(t,x,y) = \sum_{k\in \mathbb Z} \psi_k(t,y)e^{ik x},\qquad 
        (t,x,y)\in [0,T_\sigma)\times \mathbb T \times (0,1)
    \end{equation}
    such that $\psi, \partial_y \psi$ and $\partial_t \psi$ are all in  $L^\infty(0, T_\delta;\mathcal{G}^3_{\alpha(t),x} H^1_{0,y})$. 
    In particular, the following estimate on the norms holds true 
    at any time $t\in [0,T_\sigma)$:
    \begin{equation}\label{estimate:main-lemma:dypsik-Gevrey-3}
    \begin{aligned}
    &\| \psi                     (t)    \|_{\mathcal{G}^3_{\alpha(t),x} H^1_{0,y}}
    \!\!+\!
    \| \partial_y \psi          (t)    \|_{\mathcal{G}^3_{\alpha(t),x} H^1_{0,y}}
    \!\!+\!
    \| \partial^2_{ty}\psi(t)    \|_{\mathcal{G}^3_{\alpha(t),x} H^1_{0,y}}
    \!\!+\!
    \| \partial_y^2 \psi        (t)    \|_{\mathcal{G}^3_{\alpha(t),x} H^1_{0,y}}
    \!\!=\\
    &\sup_{k\in \mathbb Z}
    \!
    \Big\{
    \! 
     e^{
         \alpha(t)|k|^\frac{1}{3}
     }
     \! \big\| \psi_k(t)\! \big\|_{L^2}
    \! \Big\}
    \! +\!
    \sup_{k\in \mathbb Z}
    \!
    \Big\{
    \! 
    e^{
        \alpha(t)|k|^\frac{1}{3}
    }
    \! 
    \big\| \partial_y \psi_k(t)\! \big\|_{L^2}
    \! 
    \Big\}
    \! +\!
    \sup_{k\in \mathbb Z}
    \!
    \Big\{
    \! 
    e^{
        \alpha(t)|k|^\frac{1}{3}
    }
    \! 
    \big\| 
        \partial^2_{ty} \psi_k(t) 
        \!
    \big\|_{L^2}
    \! 
    \Big\}
    \! +\!
    \sup_{k\in \mathbb Z}
    \!
    \Big\{
    \! 
    e^{
        \alpha(t)|k|^\frac{1}{3}
    }
    \! 
    \big\| \partial_y^2 \psi_k(t)\! \big\|_{L^2}
    \! 
    \Big\}\\
    &\leq 
    C_\sigma(t)
    \Big(
        \!
        1
        \!+\! 
        \| U_{{\rm sh}} \|_{L^\infty}
        \!+\!
        \| U_{{\rm sh}}' \|_{L^\infty}
        \!+\!
        \| U_{{\rm sh}}'' \|_{L^\infty}
        \!+\! 
        \| U_{{\rm sh}}''' \|_{L^\infty}
        \!
    \Big)^2
    \!\!
    \sup_{k\in \mathbb Z}
    \Big\{
        e^{\sigma|k|^\frac{1}{3}}
        \Big(
        \| \Phi_{{\rm in}, k} \|_{L^2}
        \!+\!
        \|  u_{{\rm in}, k}       \|_{L^2}
        \!+\!
        \| u_{t,{\rm in}, k} \|_{L^2}
        \Big)
    \Big\},
    \end{aligned}
    \end{equation}
    where $C_\sigma(t) = 170\cdot \max
    \{
        1,12/{\sigma}
    \}^6t(1+t)^3e^{t}$.
\end{lemma}
\begin{proof}
If the sequence $(\psi_k)_{k\in \mathbb{Z}}$ and its derivatives satisfy the corresponding inequality \eqref{estimate:main-lemma:dypsik-Gevrey-3}, then it generates trivially a function $\psi :[0, T_\delta) \times \mathbb T \times (0,1)\to \mathbb R$ as described in \eqref{def-psi-through-Fourier-Series} , since the series converges strongly in $L^\infty(0,T_\sigma;L^2(\mathbb T\times (0,1))$ and the limit has explicit Fourier coefficients given by $(\psi_k)_{k\in \mathbb{Z}}$. 
Our main objective is therefore to prove uniquely the inequality \eqref{estimate:main-lemma:dypsik-Gevrey-3} in the coefficient $(\psi_k)_{k\in \mathbb{Z}}$.

\noindent
We first show that the function $(\partial_t +1 )\partial_y \psi_k$ and $\partial_y^2 \psi_k$ satisfies a similar inequality, namely
\begin{equation}\label{estimate:main-lemma:dtdypsik+dypsik}
    \begin{aligned}
    e^{
        \alpha(t)|k|^\frac{1}{3}
    }
    &\Big(
        \big\| (\partial_t+1)\partial_y \psi_k(t) \big\|_{L^2}
    +
        \big\| \partial_y^2 \psi_k(t) \big\|_{L^2}
    \Big)
    \leq 
    \frac{C_\sigma(t)}{5}
    \Big(
        1+ 
        \| U_{{\rm sh}} \|_{L^\infty}
        +\\
        &+
        \| U_{{\rm sh}}' \|_{L^\infty}
        +
        \| U_{{\rm sh}}'' \|_{L^\infty}
        + 
        \| U_{{\rm sh}}''' \|_{L^\infty}
    \Big)^2
    \sup_{\tilde k\in \mathbb Z}
    \Big\{
        e^{\sigma|\tilde k|^\frac{1}{3}}
        \Big(
        \| \Phi_{{\rm in}, \tilde k} \|_{L^2}
        \!+\!
        \|  u_{{\rm in}, \tilde k}       \|_{L^2}
        \!+\!
        \| u_{t,{\rm in}, \tilde k} \|_{L^2}
        \Big)
    \Big\},
   \end{aligned}
\end{equation}
for any $k\in \mathbb Z$ and $t\in [0, T_\sigma)$. When $k = 0$, then \eqref{estimate:main-lemma:dtdypsik+dypsik} is essentially a direct consequence of \eqref{est:(dt+1)dypsik-first-est-with-Gevrey-3}, which implies in particular  
\begin{equation*}
       \big\| (\partial_t+1)\partial_y \psi_0(t) \big\|_{L^2}
    +
        \big\| \partial_y^2 \psi_0(t) \big\|_{L^2}
        \leq 
        4t
        \Big\{ 
        \|  u_{{\rm in}, 0}       \|_{L^2}+ 
        \| u_{t,{\rm in}, 0} \|_{L^2}
        \Big\}.
\end{equation*}
We turn our attention therefore to $|k|\geq 1$, so that \eqref{est:(dt+1)dypsik-first-est-with-Gevrey-3} yields
\begin{align*}
    \sup_{s\in [0,t]}
    \Big\{
        \big\| (\partial_t\! +\!1)\partial_y \psi_k(s) \big\|_{L^2}
        \!+\!
        \big\| \partial_y^2 \psi_k(s)             \big\|_{L^2}
    \Big\}
    \leq 
    4t(3+\sqrt{2})
    \Big(
            1 \!+\! \| U_{{\rm sh}}' \|_{L^\infty} \!+\! \| U_{{\rm sh}} \|_{L^\infty}
    \Big)
        \Big(
        \| \Phi_{{\rm in}, k} \|_{L^2}
        \!+\!
        \|  u_{{\rm in}, k}       \|_{L^2}
        \!+\\
        +
        \| u_{t,{\rm in}, k} \|_{L^2}
        \Big)
        \frac{4\sqrt{2}}{3}
        e^t(1+t)^3
        \Big(
            1
            +
            \| U_{{\rm sh}}''' \|_{L^\infty}
            +
            \| U_{{\rm sh}}'' \|_{L^\infty}
        \Big)
        |k|^2
        \exp
        \bigg\{2^\frac{5}{6}|k|^\frac{1}{3}\Big(
        \| U_{{\rm sh}}''' \|_{L^\infty}
        +
        2\| U_{{\rm sh}}'' \|_{L^\infty}
        \Big)^\frac{1}{3}e^\frac{t}{3}t\bigg\}.
\end{align*}
Hence, we collect all terms of the shear flow $U_{\rm sh}$ within a single parenthesis, we extrapolate $e^{\sigma |k|^{1/3}}$ in front of the initial data and we choose $s = t$ on the supremum at the left-hand side
\begin{align*}
    \big\| (\partial_t +1)\partial_y \psi_k(t) \big\|_{L^2}&+
    \big\| \partial_y^2 \psi_k(t)             \big\|_{L^2}
    \leq 
    34t(1+t)^3e^{t}
    \big(
        1+ 
        \| U_{{\rm sh}} \|_{L^\infty}+ 
        \| U_{{\rm sh}}' \|_{L^\infty}+ 
        \\
        +
        \| U_{{\rm sh}}'' \|_{L^\infty}
        &+
        \| U_{{\rm sh}}''' \|_{L^\infty}
    \big)^2
    \bigg(
    \sup_{\tilde k\in \mathbb Z}
    \Big\{
        e^{\sigma|\tilde k|^\frac{1}{3}}
        \Big(
        \| \Phi_{{\rm in}, \tilde k} \|_{L^2}
        \!+\!
        \|  u_{{\rm in}, \tilde k}       \|_{L^2}
        +
        \| u_{t,{\rm in}, \tilde k} \|_{L^2}
        \Big)\!
    \Big\}
    \bigg)
    \times \\ 
    &\times 
    |k|^2
    \exp
    \Big\{
        \!\!-\sigma|k|^\frac{1}{3}\!+\!
        2^\frac{1}{3}|k|^\frac{1}{3}
        \Big(
        \| U_{{\rm sh}}''' \|_{L^\infty}
        \!+\!
        2\| U_{{\rm sh}}'' \|_{L^\infty}
        \!\Big)^\frac{1}{3}
        \!e^\frac{t}{3}t
    \Big\}.
\end{align*}
Since one has $a\leq e^{a}$, for any positive real number $a>0$, we can bound the term $|k|^2$, by means of
\begin{equation*}
    |k|^2 = \big(|k|^{\frac{1}{3}}\big)^6
    = \Big(\frac{12}{\sigma}\Big)^6
    \Big(\frac{\sigma}{12}|k|^{\frac{1}{3}}\Big)^6
    \leq 
    \Big(\frac{12}{\sigma}\Big)^6
    \Big(e^{\frac{\sigma}{12}|k|^{\frac{1}{3}}}\Big)^6
    = 
    \Big(\frac{12}{\sigma}\Big)^6
    e^{\frac{\sigma}{2}|k|^{\frac{1}{3}}}.
\end{equation*}
Hence, remarking that $ 34(12/\sigma)^6 t(1+t)^3e^t \leq C_\sigma(t)/5$, we deduce that
\begin{align*}
    &\big\| (\partial_t +1)\partial_y \psi_k(t) \big\|_{L^2} +
    \big\| \partial_y^2 \psi_k(t)             \big\|_{L^2}
    \leq 
    \frac{C_\sigma(t)}{4}
    \big(
        1+ 
        \| U_{{\rm sh}} \|_{L^\infty}+ 
        \| U_{{\rm sh}}' \|_{L^\infty}+ 
        \| U_{{\rm sh}}'' \|_{L^\infty}
        +
        \| U_{{\rm sh}}''' \|_{L^\infty}
    \big)^2\times \\
    &\times 
    \bigg(\!
    \sup_{\tilde k\in \mathbb Z}
    \Big\{
        e^{\sigma|\tilde k|^\frac{1}{3}}
        \Big(
        \| \Phi_{{\rm in}, \tilde k} \|_{L^2}
        +
        \|  u_{{\rm in}, \tilde k}       \|_{L^2}
        +
        \| u_{t,{\rm in}, \tilde k} \|_{L^2}
        \Big)
    \Big\}\!
    \bigg)
    \exp
    \Big\{
        \!-\!
        \underbrace{\Big(
        \frac{\sigma}{2}
        \!-\!
        2^\frac{1}{3}
        \Big(
        \| U_{{\rm sh}}''' \|_{L^\infty}
        \!+\!
        2\| U_{{\rm sh}}'' \|_{L^\infty}
        \!\Big)^\frac{1}{3}
        \!e^\frac{t}{3}t
        \Big)}_{= \alpha(t)}
        |k|^\frac{1}{3}
    \!\Big\},
\end{align*}
which coincides with \eqref{estimate:main-lemma:dtdypsik+dypsik}.
We now transfer the estimate \eqref{estimate:main-lemma:dtdypsik+dypsik} directly to the functions $\psi_k$, $\partial_y \psi_k$ and $\partial_t\partial_y\psi_k$ in \eqref{estimate:main-lemma:dypsik-Gevrey-3}. We begin with by developing $\partial_y \psi_k$ through
\begin{equation*}
    \partial_y \psi_k(t,y) = e^{-t}e^{t} \partial_y \psi_k(t,y) = 
    e^{-t} \int_0^t \partial_s (e^{s} \partial_y \psi_k(s,y))ds = 
     \int_0^t e^{s-t}\big[(\partial_t +1)\partial_y \psi_k\big](s,y) ds,
\end{equation*}
for any $(t,y)\in (0,T_\sigma)\times (0,1)$. Hence, by taking the $L^2$-norm on both left and right-hand sides and keeping in mind that $\alpha$ is  decreasing in time, we remark that
\begin{equation}\label{AAA}
\begin{aligned}
    &e^{\alpha(t) |k|^\frac{1}{3}}
    \| \partial_y \psi_k(t) \|_{L^2} 
    \leq 
    e^{\alpha(t) |k|^\frac{1}{3}}\int_0^t e^{s-t}\| (\partial_t +1)\partial_y \psi_k(s) \|_{L^2}
    ds 
    \leq 
    \int_0^t e^{s-t}e^{\alpha(s) |k|^\frac{1}{3}}\| (\partial_t +1)\partial_y \psi_k(s) \|_{L^2}ds\\
    &\leq \!
    \frac{C_\sigma(t)}{5}
    \Big(\!
        1\!+ \!
        \| U_{{\rm sh}} \|_{L^\infty}
        \!+\!
        \| U_{{\rm sh}}' \|_{L^\infty}
        \!+\!
        \| U_{{\rm sh}}'' \|_{L^\infty}
        \!+\!\| U_{{\rm sh}}''' \|_{L^\infty}
    \!\Big)^2
    \sup_{\tilde k\in \mathbb Z}
    \Big\{
        e^{\sigma|\tilde k|^\frac{1}{3}}
        \Big(\!
        \| \Phi_{{\rm in}, \tilde k} \|_{L^2}
        \!+\!
        \|  u_{{\rm in}, \tilde k}       \|_{L^2}
        \!+\!
        \| u_{t,{\rm in}, \tilde k} \|_{L^2}
        \!\Big)
    \Big\},
    \end{aligned}
\end{equation}
where we have used $ \int_0^t e^{s-t}ds = 1-e^{-t}<1$.
Invoking the Poincar\'e inequality $\| \psi_k \|_{L^2}\leq \|\partial_y \psi_k \|_{L^2} $, it is easy at this stage to check that
\begin{align*}
    &\sup_{k\in \mathbb Z}
    \!
    \Big\{
    \! 
     e^{
         \alpha(t)|k|^\frac{1}{3}
     }
     \! \big\| \psi_k(t)\! \big\|_{L^2}
    \! \Big\}
    \! +\!
    \sup_{k\in \mathbb Z}
    \!
    \Big\{
    \! 
    e^{
        \alpha(t)|k|^\frac{1}{3}
    }
    \! 
    \big\| \partial_y \psi_k(t)\! \big\|_{L^2}
    \! 
    \Big\}
    \! +\!
    \sup_{k\in \mathbb Z}
    \!
    \Big\{
    \! 
    e^{
        \alpha(t)|k|^\frac{1}{3}
    }
    \! 
    \big\| 
        \partial^2_{ty} \psi_k(t) 
        \!
    \big\|_{L^2}
    \! 
    \Big\}
    \! +\!
    \sup_{k\in \mathbb Z}
    \!
    \Big\{
    \! 
    e^{
        \alpha(t)|k|^\frac{1}{3}
    }
    \! 
    \big\| \partial_y^2 \psi_k(t)\! \big\|_{L^2}
    \! 
    \Big\}\\
    &
    \leq 
    2\sup_{k\in \mathbb Z}
    \Big\{
    e^{
        \alpha(t)|k|^\frac{1}{3}
    }
    \big\| \partial_y \psi_k(t)  \big\|_{L^2}
    \Big\}
    +
    \sup_{k\in \mathbb Z}
    \Big\{
    e^{
        \alpha(t)|k|^\frac{1}{3}
    }
    \big\| 
        \partial^2_{ty} \psi_k(t) 
    \big\|_{L^2}
    \Big\}
     +
    \sup_{k\in \mathbb Z}
    \Big\{
    e^{
        \alpha(t)|k|^\frac{1}{3}
    }
    \big\| \partial_y^2 \psi_k(t) \big\|_{L^2}
    \Big\}
    \\
    &\leq 
    3\sup_{k\in \mathbb Z}
    \Big\{
    e^{
        \alpha(t)|k|^\frac{1}{3}
    }
    \big\| \partial_y \psi_k(t) \big\|_{L^2} 
    \Big\}
    +
    \sup_{k\in \mathbb Z}
    \!
    \Big\{
    e^{
        \alpha(t)|k|^\frac{1}{3}
    }
    \big\| 
        (\partial_t +1)\partial_{y} \psi_k(t) 
    \big\|_{L^2}
    \Big\}+
    \sup_{k\in \mathbb Z}
    \Big\{
    e^{
        \alpha(t)|k|^\frac{1}{3}
    }
    \big\| \partial_y^2 \psi_k(t) \big\|_{L^2}
    \Big\}.
\end{align*}
We finally couple the last inequality together with \eqref{estimate:main-lemma:dtdypsik+dypsik} and \eqref{AAA}, which finally implies the estimate \eqref{estimate:main-lemma:dypsik-Gevrey-3}. This concludes the proof of the lemma.
\end{proof}
\subsection{Transferring Gevrey 3 to the velocity field}\label{sec:transferring-G3-to-u}
We shall now transfer the Gevrey regularity from the function $\psi$ of Lemma \ref{lemma:Gevrey3-psi} to a solution $u$ of the original hyperbolic Prandtl equation \eqref{eq:Prandtl_hyperbolic}. 
\begin{prop}\label{prop:transferring-Gevrey-to-uk}
    Assume that $u_{\rm in }$ is in $\mathcal{G}_{\sigma,x}^3H^1_{0,y}$, while $u_{t,{\rm in} }$ is in $\mathcal{G}_{x, \sigma}^3L^2_y$, for a given $\sigma>0$.
    Let $T_\sigma>0$ and $\beta,\,\gamma:[0, T_\sigma)\to \mathbb R$ be as in \Cref{main-thm}.   
    Then the sequence $(u_k)_{k\in \mathbb N}$ generates a function $u: [0,T_\sigma)\times \mathbb T\times (0,1)\to \mathbb R$ through the inverse Fourier transform
    \begin{equation}\label{prop2.5:u-Fourier}
        u(t,x,y) = \sum_{k\in \mathbb Z} u_k(t,y)e^{ik x},\qquad 
        (t,x,y)\in [0,T_\sigma)\times \mathbb T \times (0,1),
    \end{equation}
    such that $u\in L^\infty(0, T_\sigma;\mathcal{G}_{\beta(t), x}H^1_{0,y})$ and $\partial_t u \in L^\infty(0, T_\sigma;\mathcal{G}_{\gamma(t), x}L^2_y)$. 
    In particular, the following estimate holds true at any time $t\in [0,T_\sigma)$:
    \begin{equation}\label{estimate:main-lemma:uk-Gevrey-3}
    \begin{aligned}
    \|  u(t)                &\|_{\mathcal{G}^3_{\beta(t),x}L^2_y}+
    \|  \partial_yu(t)      \|_{\mathcal{G}^3_{\beta(t),x}L^2_y}+
    \|  \partial_tu(t)      \|_{\mathcal{G}^3_{\gamma(t),x}L^2_y}
    \\
    &=
    \sup_{k\in \mathbb Z}
    \Big\{
        e^{
          \beta(t)|k|^\frac{1}{3}
        }
       \big\| u_k(t) \big\|_{L^2}
    \Big\}
    +
    \sup_{k\in \mathbb Z}
    \Big\{
     e^{
        \beta(t)|k|^\frac{1}{3}
    }
        \big\| \partial_y u_k(t) \big\|_{L^2}
    \Big\}    
    +
    \sup_{k\in \mathbb Z}
    \Big\{
    e^{
        \gamma(t)|k|^\frac{1}{3}
    }
    \big\| \partial_t u_k(t) \big\|_{L^2}
    \Big\}\\
    &\leq 
    \tilde C_\sigma(t)
    \Big(
        1+ 
        \| U_{{\rm sh}} \|_{L^\infty}
        +
        \| U_{{\rm sh}}' \|_{L^\infty}
        +
        \| U_{{\rm sh}}'' \|_{L^\infty}
        + 
        \| U_{{\rm sh}}''' \|_{L^\infty}
    \Big)^3
    \Big(
        \|  u_{{\rm in}}       \|_{\mathcal{G}_{\sigma, x}^3 H^1_{0,y}}
        +    
        \| u_{t,{\rm in}} \|_{\mathcal{G}_{\sigma, x}^3L^2_y}
    \Big),
    \end{aligned}
    \end{equation}
    where $\tilde C_\sigma(t) = 
    10^4
    \max
    \{
        1,12/{\sigma}
    \}^{15} (1+t)^5e^{t}$.
\end{prop}
\begin{proof}
Similarly as in \Cref{lemma:Gevrey3-psi}, we focus this entire proof to show the estimate \eqref{estimate:main-lemma:uk-Gevrey-3} on the sequence $(u_k)_{k\in \mathbb Z}$. The solution $u$ is then automatically determined by the Fourier series \eqref{prop2.5:u-Fourier}.

\noindent
We  fix momentarily the frequency $k \in \mathbb Z$. If the initial data $(u_{{\rm in}, k},\,u_{{\rm in}, k,t})$ are identically null, then the solution $u_k$ of \eqref{eq:Prandtl-in-uk-Fourier} is identically null (since the equation in \eqref{eq:Prandtl-in-uk-Fourier} for a fixed $k\in \mathbb Z$ is linear, hyperbolic and damped) and the inequality \eqref{estimate:main-lemma:uk-Gevrey-3} is automatically satisfied. We shall thus focus this proof to the case $(u_{{\rm in}, k},\,u_{{\rm in}, k,t})\neq (0,0)$.

\noindent 
We begin with by setting the function $f_k := \big( (\partial_{t} + 1)(\partial_t  + ik U_{{\rm sh}} )  -\partial_y^2\big)
    \partial_y
    \psi_k$. Thanks  to  \eqref{eq:psik-not-compact}, $f_k$ satisfies 
\begin{equation}\label{eq:fk-and-uk}
    f_k(t,y) = ik U_{\rm sh}'(y) (\partial_t +1) \psi_k(t,y) + \partial_y \Phi_k(t,y)
    = 
    ik U_{\rm sh}'(y) (\partial_t +1) \psi_k(t,y) + u_k(t,y),
\end{equation}    
for any $ (t,y) \in [0, T_\sigma)\times (0,1)$. Because of the boundary conditions on $\psi_k$ and $u_k$, the function $f_k$ fulfills homogeneous Dirichlet conditions $f_{k|y = 0,1} = 0 $. Furthermore its initial data are determined by
\begin{equation}\label{initial-condition-on-fk}
    f_k(0,y) = u_{{\rm in} , k}(y)
    \quad 
    \partial_t f_k(0,y)
    = 
    ik U_{{\rm sh}}'(y)
    \partial_t^2 \psi_{k}(0, y) + 
    u_{t,{\rm in} , k}(y)
     = 
    ik U_{{\rm sh}}'(y)
    \Phi_{{\rm in},k}(y) + 
    u_{t,{\rm in} , k}(y)
\end{equation}
for any $y\in (0,1)$. Thanks to identity \eqref{eq:d_ypsik-final-form}, we remark moreover that $f_k$ is also solution of
    \begin{equation}\label{eq:d_ypsik-final-form2}
\begin{aligned}
     \big( (\partial_{t} + 1)(\partial_t  + ik U_{{\rm sh}} )  -\partial_y^2\big)f_k
    = 
   [ik U_{{\rm sh}}'  , \partial_{y}^2]\big( (\partial_t + 1)\psi_k\big).
\end{aligned}
\end{equation}
Starting from \eqref{eq:d_ypsik-final-form2}, we aim to develop some suitable estimates on the $L^2$-norm of $(\partial_t+1)f_k$ and $\partial_y f_k$, which we will eventually transfer to $u_k$, $\partial_y u_k$ and $\partial_t u_k$, making use of \eqref{eq:fk-and-uk}. We shall observe that we have now a complete control on the right-hand side of \eqref{eq:d_ypsik-final-form2}, because of the uniform estimates given by  \Cref{lemma:Gevrey3-psi}.

\noindent
We multiply \eqref{eq:d_ypsik-final-form2} with the conjugate $\overline{(\partial_t +1)f_k}$, then we integrate along $(0,t) \times (0,1)$ for a time $t\in [0,T_\sigma)$ and finally we calculate the real part ${\mathfrak{Re}}$ of the result. This leads in particular to
\begin{equation}\label{eq:first-equation-of-(dt+1)fk}
\begin{aligned}
    \frac{1}{2}\big\|  
        (\partial_t +1)
        f_k(t) 
    \big\|_{L^2}^2
    +
    \frac{1}{2}
    \big\|  
        \partial_y
        f_k(t) 
    \big\|_{L^2}^2
    +
    \int_0^t 
    \big\|  
        \partial_y 
        f_k(s) 
    \big\|_{L^2}^2ds
    = 
    \frac{1}{2}\big\|  
        u_{{\rm in} , k}
        +
        ik U_{{\rm sh}}'
        \Phi_{{\rm in},k} + 
        u_{t,{\rm in} , k}
    \big\|_{L^2}^2
    +\\
    +
    \frac{1}{2}\big\|  
        \partial_y u_{{\rm in} , k}
    \big\|_{L^2}^2
    +
    {\mathfrak{Re}}
    \bigg(
    ik 
    \int_0^t \int_0^1 
    [ik U_{{\rm sh}}'  , \partial_{y}^2]\big( (\partial_t + 1)\psi_k\big)
    \overline{(\partial_t +1)f_k}
    dy
    ds
    \bigg),
\end{aligned}
\end{equation}
where we have used the initial conditions in \eqref{initial-condition-on-fk}, as well as
\begin{equation*}
    {\mathfrak{Re}}
    \bigg( 
    ik
    \int_0^t \int_0^1 U_{\rm sh}(y) |(\partial_t +1) f_k(s,y)|^2dyds
    \bigg) = 0.
\end{equation*}
Multiplying \eqref{eq:first-equation-of-(dt+1)fk} by $2$ and applying the Cauchy-Schwarz inequality on the last integral, we obtain
\begin{equation}\label{est-fk-before-commutator}
\begin{aligned}
    \big\|  
        (\partial_t +1)
        f_k(t) 
    \big\|_{L^2}^2+
    \big\|  
        \partial_y
        f_k(t) 
    \big\|_{L^2}^2
    +
    2
    \int_0^t
    \big\|  
        \partial_y 
        f_k(s) 
    \big\|_{L^2}^2ds
    \leq 
    \big\|  
        u_{{\rm in} , k}
        +
        ik U_{{\rm sh}}'
        \Phi_{{\rm in},k} + 
        u_{t,{\rm in} , k}
    \big\|_{L^2}^2+ 
        \\
   +
    \| \partial_y u_{{\rm in}, k} \|_{L^2}^2
    +
    2|k|
    \int_0^t
    \big\|  
        [ U_{{\rm sh}}'  , \partial_{y}^2]\big( (\partial_t + 1)\psi_k(s)\big)
    \big\|_{L^2}
    \| (\partial_t +1)f_k(s) \|_{L^2}ds.
\end{aligned}
\end{equation}
We next address the last integral in \eqref{eq:first-equation-of-(dt+1)fk} with the commutator $[ U_{{\rm sh}}'  , \partial_{y}^2]( (\partial_t + 1)\psi_k)$. First, we remark
\begin{align*}
    \big\|  
        [ U_{{\rm sh}}'  , \partial_{y}^2]\big( (\partial_t + 1)\psi_k(s)\big)
    \big\|_{L^2}
    &= 
    \big\|  
        U_{{\rm sh}}' \partial_{y}^2\big( (\partial_t + 1)\psi_k(s) \big) - 
        \partial_{y}^2\big(U_{{\rm sh}}'(\partial_t + 1)\psi_k(s)\big)
    \big\|_{L^2}\\
    &= 
    \big\|  
        U_{{\rm sh}}''' (\partial_t + 1)\psi_k(s) +2 
        U_{{\rm sh}}'' (\partial_t + 1)\partial_{y}\psi_k(s)
    \big\|_{L^2}
    \\
    &\leq 
    \| U_{{\rm sh}}''' \|_{L^\infty}
   \|(\partial_t +  1)\psi_k(s)\|_{L^2}
    +2
    \| U_{{\rm sh}}'' \|_{L^\infty}
    \|(\partial_t +  1)\partial_y\psi_k(s)\|_{L^2}
    \\
    &\leq 
    \big(
    \| U_{{\rm sh}}''' \|_{L^\infty}
    +2
    \| U_{{\rm sh}}'' \|_{L^\infty}
    \big)
    \|(\partial_t +  1)\partial_y\psi_k(s)\|_{L^2}
\end{align*}
Hence, we plug this last inequality into \eqref{est-fk-before-commutator}, we take the Supermum within the time interval $[0,t]$ and finally we divide the result by 
$\sup_{s\in [0,t]}( \| (\partial_t +1)f_k(s) \|_{L^2}^2+\| (\partial_t +1)f_k(s) \|_{L^2}^2)^{1/2} $ (which is not null, since the initial data are not all zero). This leads to
\begin{align*}
   &
   \sup_{s\in [0,t]}
   \Big(
    \|  
        (\partial_t +1)
        f_k(s) 
    \|_{L^2}
    +
    \| 
        \partial_y f_k(s)
    \|_{L^2}
    \Big)
    \leq 
    \sqrt{2}
    \sup_{s\in [0,t]}
    \Big(
    \|  
        (\partial_t +1)
        f_k(s) 
    \|_{L^2}^2
    +
    \| 
        \partial_y f_k(s)
    \|_{L^2}^2
    \Big)^\frac{1}{2}
    \\
    &\leq 
    \sqrt{2}\frac{
    \sup_{s\in [0,t]}
    \Big(
    \|  
        (\partial_t +1)
        f_k(s) 
    \|_{L^2}^2
    +
    \| 
        \partial_y f_k(s)
    \|_{L^2}^2
    \Big)
    }
    {
    \sup_{s\in [0,t]}
    \Big(
    \|  
        (\partial_t +1)
        f_k(s) 
    \|_{L^2}^2
    +
    \| 
        \partial_y f_k(s)
    \|_{L^2}^2
    \Big)^\frac{1}{2}
    }
    \\
    &\leq 
    \sqrt{2}
    \frac{
        \big\|  
            u_{{\rm in} , k}
            +
            ik U_{{\rm sh}}'
            \Phi_{{\rm in},k} + 
            u_{t,{\rm in} , k}
        \big\|_{L^2}^2
        +
        \| \partial_y u_{{\rm in}, k} \|_{L^2}^2
    }
    {
        \sup_{s\in [0,t]}
        \big(
        \|  
            (\partial_t +1)
            f_k(s) 
        \|_{L^2}^2
        +
        \| \partial_y f_k(s) \|_{L^2}^2
        \big)^\frac{1}{2}
    }
    +\\
    &+
    2\sqrt{2}
    \big(
    \| U_{{\rm sh}}''' \|_{L^\infty}
    +2
    \| U_{{\rm sh}}'' \|_{L^\infty}
    \big)
    |k|
    \int_0^t
     \|(\partial_t +  1)\partial_y\psi_k(s)\|_{L^2}
     \frac{
     \|(\partial_t +  1)f_k(s)\|_{L^2}
    }
    {
        \sup_{s\in [0,t]}
        \big(
        \|  
            (\partial_t +1)
            f_k(s) 
        \|_{L^2}^2
        +
        \| \partial_y f_k(s) \|_{L^2}^2
        \big)^\frac{1}{2}
    }
     ds.
\end{align*}
At $s = 0$, 
the functions $ (\partial_t +1)f_k(s) $ and $\partial_y f_k(s)$ coincide with $u_{{\rm in} , k}+ik U_{{\rm sh}}'\Phi_{{\rm in},k} +  u_{t,{\rm in} , k} $ and $\partial_y u_{{\rm in}, k}$, respectively. We deduce hence that
\begin{equation}\label{est:fk-almost-there}
\begin{aligned}
   \sup_{s\in [0,t]}
   \Big(
    \|  
        (\partial_t +1)
        f_k(s) 
    \|_{L^2}
    +
    \| 
        \partial_y f_k(s)
    \|_{L^2}
    \Big)
    \leq 
    \sqrt{2}
    \Big(
        \big\|  
            u_{{\rm in} , k}
            +
            ik U_{{\rm sh}}'
            \Phi_{{\rm in},k} + 
            u_{t,{\rm in} , k}
        \big\|_{L^2}^2
        +
        \| \partial_y u_{{\rm in}, k} \|_{L^2}^2
    \Big)^\frac{1}{2}
    +\\
    +
    2
    \sqrt{2}
    \big(
    \| U_{{\rm sh}}''' \|_{L^\infty}
    \!+\!
    2
    \| U_{{\rm sh}}'' \|_{L^\infty}
    \big)
    |k|
    \int_0^t\!\!
     \|(\partial_t\!+\!
      1)\partial_y\psi_k(s)\|_{L^2}
     ds
    \leq 
    \sqrt{2}\Big(
        \|  u_{{\rm in} , k} \|_{L^2}
        \!+\!
        |k| 
        \| U_{{\rm sh}}' \|_{L^\infty}
        \| \Phi_{{\rm in},k} \|_{L^2}
        + \\ 
        +
        \|
            u_{t,{\rm in} , k}
        \|_{L^2}
        +
        \| \partial_y u_{{\rm in}, k} \|_{L^2}
    \Big)
    +
    2
    \sqrt{2}
    \big(
    \| U_{{\rm sh}}''' \|_{L^\infty}
    +2
    \| U_{{\rm sh}}'' \|_{L^\infty}
    \big)
    |k|
    \int_0^t
     \|(\partial_t +  1)\partial_y\psi_k(s)\|_{L^2}
     ds.
\end{aligned}
\end{equation}
We are now in the condition to reveal the uniform estimates of the Gevrey-three regularity  on the sequences $((\partial_t +1)f_k)_{k\in\mathbb Z}$ and $(\partial_y f_k)_{k\in\mathbb Z}$, by establishing the corresponding exponential growth on the modes $|k|$. To this end, we take advantage of \Cref{lemma:Gevrey3-psi}, so that the last integral in \eqref{est:fk-almost-there} enables
\begin{align*}
    \int_0^t
     \|(\partial_t +  1)\partial_y\psi_k(s)\|_{L^2}
     ds
     \leq 
    \bigg(
    \int_0^t
    C_\sigma(s)e^{-\alpha(s) |k|^\frac{1}{3}}
    ds
    \bigg)
    \Big(
        1+ 
        \| U_{{\rm sh}} \|_{L^\infty}
        +
        \| U_{{\rm sh}}' \|_{L^\infty}
        +
        \| U_{{\rm sh}}'' \|_{L^\infty}
        + \\+
        \| U_{{\rm sh}}''' \|_{L^\infty}
    \Big)^2
    \sup_{\tilde k\in \mathbb Z}
    \Big\{
        e^{\sigma|\tilde k|^\frac{1}{3}}
        \Big(
        \| \Phi_{{\rm in}, \tilde k} \|_{L^2}
        \!+\!
        \|  u_{{\rm in}, \tilde k}       \|_{L^2}
        \!+\!
        \| u_{t,{\rm in}, \tilde k} \|_{L^2}
        \!+\!
        \| \partial_y u_{{\rm in}, \tilde k} \|_{L^2}
        \Big)
    \Big\},
\end{align*}
where we recall that $C_\sigma$ in \Cref{lemma:Gevrey3-psi} is defined as $C_\sigma(t) = 170\max\{1,12/{\sigma}\}^6t(1+t)^3e^{t}$, while the radius $\alpha(t) =  \sigma/2-2^\frac{5}{6}(\| U_{{\rm sh}}''' \|_{L^\infty}+
2\| U_{{\rm sh}}'' \|_{L^\infty})^\frac{1}{3}e^{t/3}t$ (which satisfies $\alpha(t)\geq 7\sigma/8>0$, for any $t\in [0,T_\sigma)$). Since $s\in [0,t]\to C_\sigma(s)e^{-\alpha(s) |k|^\frac{1}{3}}$ is a non-decreasing function, it can be bound by its value at $s = t$, so that
\begin{align*}
    \int_0^t
    \!
     \|(\partial_t +  1)\partial_y\psi_k(s)\|_{L^2}
     ds
     \leq
    t C_\sigma(t) e^{-\alpha(t) |k|^\frac{1}{3}}
    \Big(
        1+ 
        \| U_{{\rm sh}} \|_{L^\infty}
        +
        \| U_{{\rm sh}}' \|_{L^\infty}
        +
        \| U_{{\rm sh}}'' \|_{L^\infty}
        + \\ 
        +
        \| U_{{\rm sh}}''' \|_{L^\infty}
    \Big)^2
    \sup_{\tilde k\in \mathbb Z}
    \Big\{
        e^{\sigma|\tilde k|^\frac{1}{3}}
        \Big(
        \| \Phi_{{\rm in}, \tilde k} \|_{L^2}
        \!+\!
        \|  u_{{\rm in}, \tilde k}       \|_{L^2}
        \!+\!
        \| u_{t,{\rm in}, \tilde k} \|_{L^2}
        \!+\!
        \| \partial_y u_{{\rm in}, \tilde k} \|_{L^2}
        \Big)
    \Big\}.
\end{align*}
We thus couple this last inequality together with \eqref{est:fk-almost-there}, which guarantees 
\begin{align*}
        \|  
            (\partial_t +1)
            f_k(t) 
        \|_{L^2}
        +
        \|  
            \partial_y 
            f_k(t) 
        \|_{L^2}
    \leq
     \sqrt{2}\Big(
        \|  u_{{\rm in} , k} \|_{L^2}
        +
        |k| 
        \| U_{{\rm sh}}' \|_{L^\infty}
        \!
        \| \Phi_{{\rm in},k} \|_{L^2}
        +
        \|
            u_{t,{\rm in} , k}
        \|_{L^2}
        + 
        \| \partial_y u_{{\rm in}, k} \|_{L^2}
    \Big)
    +\\+
    4
    \sqrt{2}
    \big(
    \| U_{{\rm sh}}''' \|_{L^\infty}
    +
    \| U_{{\rm sh}}'' \|_{L^\infty}
    \big)
    |k|
     t C_\sigma(t) e^{-\alpha(t) |k|^\frac{1}{3}}
    \Big(
        1+ 
        \| U_{{\rm sh}} \|_{L^\infty}
        +
        \| U_{{\rm sh}}' \|_{L^\infty}
        + \\ +
        \| U_{{\rm sh}}'' \|_{L^\infty}
        +
        \| U_{{\rm sh}}''' \|_{L^\infty}
    \Big)^2
    \sup_{\tilde k\in \mathbb Z}
    \Big\{
        e^{\sigma|\tilde k|^\frac{1}{3}}
        \Big(
        \| \Phi_{{\rm in}, \tilde k} \|_{L^2}
        \!+\!
        \|  u_{{\rm in}, \tilde k}       \|_{L^2}
        \!+\!
        \| u_{t,{\rm in},\tilde  k} \|_{L^2}
        \!+\!
        \| \partial_y u_{{\rm in}, k} \|_{L^2}
        \Big)
    \Big\}.
\end{align*}
Finally, to obtain our uniform estimates, we multiply both left and right-hand sides with $e^{\beta(t)|k|^{1/3}}$ (where $\beta(t)= \alpha(t)-\sigma/4 $ is defined in \eqref{radii-of-Gevrey-beta-gamma}), to obtain 
\begin{align*}
        e^{\beta(t)|k|^\frac{1}{3}}
        \Big(
        \|  
            (\partial_t +1)
            f_k(t) 
        \|_{L^2}
        +
        \|  
            \partial_y 
            f_k(t) 
        \|_{L^2}
        \Big)
    \leq
     \sqrt{2}
     e^{\beta(t)|k|^\frac{1}{3}}
     \Big(
        \|  u_{{\rm in} , k} \|_{L^2}
        +
        |k| 
        \| U_{{\rm sh}}' \|_{L^\infty}
        \| \Phi_{{\rm in},k} \|_{L^2}
        +
        \|
            u_{t,{\rm in} , k}
        \|_{L^2}
        +\\+ 
        \| \partial_y u_{{\rm in}, k} \|_{L^2}
    \Big)
    +
    4
    \sqrt{2}
    e^{-\frac{\sigma}{4}|k|^\frac{1}{3}}
    \big(
    \| U_{{\rm sh}}''' \|_{L^\infty}
    +
    \| U_{{\rm sh}}'' \|_{L^\infty}
    \big)
    |k|
     t C_\sigma(t)
    \Big(
        1+ 
        \| U_{{\rm sh}} \|_{L^\infty}
        +
        \| U_{{\rm sh}}' \|_{L^\infty}
        + \\ +
        \| U_{{\rm sh}}'' \|_{L^\infty}
        +
        \| U_{{\rm sh}}''' \|_{L^\infty}
    \Big)^2
    \sup_{\tilde k\in \mathbb Z}
    \Big\{
        e^{\sigma|\tilde k|^\frac{1}{3}}
        \Big(
        \| \Phi_{{\rm in}, \tilde k} \|_{L^2}
        \!+\!
        \|  u_{{\rm in}, \tilde k}       \|_{L^2}
        \!+\!
        \| u_{t,{\rm in},\tilde  k} \|_{L^2}
        \!+\!
        \| \partial_y u_{{\rm in}, k} \|_{L^2}
        \Big)
    \Big\}.
\end{align*}
The left-hand side already reveals the Gevrey-three norm on $((\partial_t +1)f_k)_{k\in \mathbb Z}$ and $(\partial_y f_k)_{k\in \mathbb Z}$. We shall however provide a uniform estimate of the right-hand side, with respect to the frequencies $k\in \mathbb Z$.
To this end, we first observe that 
\begin{equation}\label{ineq:k-est-with-exp}
    |k| =  
    \bigg(
    \frac{4}{\sigma}
    \bigg)^3
    \Big(
        \frac{\sigma}{4}|k|^{\frac{1}{3}}
    \Big)^3 
    \leq 
    \bigg(
    \frac{4}{\sigma}
    \bigg)^3 e^{\frac{3\sigma}{4}|k|^{\frac{1}{3}}},
    \qquad 
    |k| =  
    \bigg(
    \frac{12}{\sigma}
    \bigg)^3
    \Big(
        \frac{\sigma}{12}|k|^{\frac{1}{3}}
    \Big)^3 
    \leq 
    \bigg(
    \frac{12}{\sigma}
    \bigg)^3 e^{\frac{\sigma}{4}|k|^{\frac{1}{3}}}.
\end{equation}
Therefore, since $1\leq\max\{1, 4/\sigma\}^3e^{3\sigma |k|^{1/3}/4} $,  we gather that
\begin{align*}
    e^{\beta(t)|k|^\frac{1}{3}}
     \Big(
        \|  
            (\partial_t \!+\!1)
            f_k(t) 
        \|_{L^2}
        \!+\!
        \|  
            \partial_y 
            f_k(t) 
        \|_{L^2}
    \Big)
    \leq
    \sqrt{2}
     e^{\big(\beta(t)+\frac{3\sigma}{4}\big)|k|^\frac{1}{3}}
    \!
     \max
    \Big\{
        1, \frac{4}{\sigma}
    \Big\}^3
     \Big(
        \|  u_{{\rm in} , k} \|_{L^2}
        \!+\! 
        \| U_{{\rm sh}}' \|_{L^\infty}
        \| \Phi_{{\rm in},k} \|_{L^2}
        \!+\!\\ +
        \|
            u_{t,{\rm in} , k}
        \|_{L^2}
        +
        \| \partial_y u_{{\rm in}, k} \|_{L^2}
    \Big)
    +
    4
    \sqrt{2}
    e^{-\frac{\sigma}{4}|k|^\frac 13 }
    \big(
    \| U_{{\rm sh}}''' \|_{L^\infty}
    +
    \| U_{{\rm sh}}'' \|_{L^\infty}
    \big)
    \bigg(
    \frac{12}{\sigma}
    \bigg)^3
    e^{\frac{\sigma}{4}|k|^\frac 13 }
    t C_\sigma(t)
    \Big(
        1+ 
        \| U_{{\rm sh}} \|_{L^\infty}
        +\\+
        \| U_{{\rm sh}}' \|_{L^\infty}
        + 
        \| U_{{\rm sh}}'' \|_{L^\infty}
        +
        \| U_{{\rm sh}}''' \|_{L^\infty}
    \Big)^2
    \sup_{\tilde k\in \mathbb Z}
    \Big\{
        e^{\sigma|\tilde k|^\frac{1}{3}}
        \Big(
        \| \Phi_{{\rm in}, \tilde k} \|_{L^2}
        \!+\!
        \|  u_{{\rm in}, \tilde k}       \|_{L^2}
        \!+\!
        \| u_{t,{\rm in}, \tilde k} \|_{L^2}
        \!+\!
        \| \partial_y u_{{\rm in}, \tilde k} \|_{L^2}
        \Big)
    \Big\}.
\end{align*}
This provides the required uniform estimate in $k\in \mathbb Z$, since $3\sigma/4 + \beta(t) =\alpha(t) + \sigma/2 \leq \sigma $ for any $t\in [0, T_\sigma)$, hence
\begin{equation}\label{ineq:est-on-dt+1fk-dyfk-final}
\begin{aligned}
    e^{\beta(t)|k|^\frac{1}{3}}
     \Big(
        \|  
            (\partial_t +1)
            f_k(t) 
        \|_{L^2}
        +
        \|  
            \partial_y 
            f_k(t) 
        \|_{L^2}
    \Big)
    \leq 
    C_{2,\sigma}(t)
    \Big(
        1+ 
        \| U_{{\rm sh}} \|_{L^\infty}
        +
        \| U_{{\rm sh}}' \|_{L^\infty}
        +
        \| U_{{\rm sh}}'' \|_{L^\infty}
        + \\+
        \| U_{{\rm sh}}''' \|_{L^\infty}
    \Big)^3
    \sup_{\tilde k\in \mathbb Z}
    \Big\{
        e^{\sigma|\tilde k|^\frac{1}{3}}
        \Big(
        \| \Phi_{{\rm in}, \tilde k} \|_{L^2}
        +
        \|  u_{{\rm in}, \tilde k}       \|_{L^2}
        +
        \| u_{t,{\rm in}, \tilde k} \|_{L^2}
        +
        \| \partial_y u_{{\rm in}, \tilde k} \|_{L^2}
        \Big)
    \Big\},
\end{aligned}
\end{equation}
where the function $t\in (0, T_\sigma) \to  C_{2,\sigma}(t)$ is now defined by
\begin{equation}\label{def:C2sigma}
\begin{aligned}
     C_{2,\sigma}(t)
    &:=
    \sqrt{2}
    \max
    \Big\{
        1, \frac{4}{\sigma}
    \Big\}^3+ 
    4
    \sqrt{2}
    \bigg( \frac{12}{\sigma}\bigg)^3
    t C_\sigma (t)\\
    &=
    \sqrt{2}
    \max
    \Big\{
        1, \frac{4}{\sigma}
    \Big\}^3
    +
    4
    \sqrt{2}
    \bigg( \frac{12}{\sigma}\bigg)^3
    170\cdot \max
    \bigg\{
        1,\frac{12}{\sigma}
    \bigg\}^6 t^2(1+t)^3e^{t}.
\end{aligned}
\end{equation}
We now take advantage of \eqref{ineq:est-on-dt+1fk-dyfk-final}, in order to transfer the corresponding Gevrey-three estimates to the sequences $(u_k)_{k\in \mathbb Z}$, $(\partial_y u_k)_{k\in \mathbb Z}$ and 
$(\partial_t u_k)_{k\in \mathbb Z}$. We begin with $\partial_y u_k$ and we invoke identity \eqref{eq:fk-and-uk}, which ensures that $\partial_y u_k = \partial_y f_k - ik U_{\rm sh}''(\partial_t +1)\psi_k -ik U_{\rm sh}'(\partial_t +1)\partial_y\psi_k$. Hence
\begin{align*}
     e^{\beta(t) |k|^\frac{1}{3}}
     \| \partial_y u_k (t) \|_{L^2}
     &\leq 
     e^{\beta(t) |k|^\frac{1}{3}}
     \Big(
        \| \partial_y f_k \|_{L^2}
        +
        |k|
        \| U_{\rm sh}'  \|_{L^\infty}
        \| (\partial_t +1) \partial_y \psi_k \|_{L^2}
        +
        |k|
        \| U_{\rm sh}''  \|_{L^\infty}
        \| (\partial_t +1)  \psi_k \|_{L^2}
     \Big),
\end{align*}
which we couple together with \eqref{ineq:k-est-with-exp}, the Poincar\'e inequality $\| (\partial_t+1)\psi_k \|_{L^2}\leq \| (\partial_t+1)\partial_y \psi_k \|_{L^2}$ and the relation $\alpha(t) = \beta(t)+\sigma/4$, to obtain
\begin{align*}
     e^{\beta(t) |k|^\frac{1}{3}}
     \| \partial_y u_k (t) \|_{L^2}
     &\leq e^{\beta(t) |k|^\frac{1}{3}}
     \bigg(
        \| \partial_y f_k \|_{L^2}
        +
        \Big(\frac{12}{\sigma}\Big)^3
        e^{\frac{\sigma}{4}|k|^\frac{1}{3}}
        \big(
            \| U_{\rm sh}'  \|_{L^\infty}
            +
            \| U_{\rm sh}''  \|_{L^\infty}
        \big)
        \| (\partial_t +1) \partial_y \psi_k \|_{L^2}
     \bigg)\\
     &\leq 
    e^{\beta(t) |k|^\frac{1}{3}}
    \| \partial_y f_k \|_{L^2}
    +
    \Big(\frac{12}{\sigma}\Big)^3
    e^{\alpha(t) |k|^\frac{1}{3}}
    \big(
        \| U_{\rm sh}'  \|_{L^\infty}
        +
        \| U_{\rm sh}''  \|_{L^\infty}
    \big)
    \| (\partial_t +1) \partial_y \psi_k \|_{L^2}.
\end{align*}
Thanks to \eqref{ineq:est-on-dt+1fk-dyfk-final} and the Poincar\'e inequality $\| u_k \|_{L^2}\leq \| \partial_y u_k \|_{L^2}$, we establish therefore the uniform estimate
\begin{equation}\label{est:ineq-on-uk-dyuk}
\begin{aligned}
    e^{\beta(t) |k|^\frac{1}{3}}
    \Big(
         \| u_k (t) \|_{L^2}
         +
         \| \partial_y u_k (t) \|_{L^2}
    \Big)
     \leq 
     \Big(  C_{2,\sigma}(t) + \Big(\frac{12}{\sigma}\Big)^3C_\sigma(t) \Big)
    \Big(
        1+ 
        \| U_{{\rm sh}} \|_{L^\infty}
        +
        \| U_{{\rm sh}}' \|_{L^\infty}
        +\\+
        \| U_{{\rm sh}}'' \|_{L^\infty}
        + 
        \| U_{{\rm sh}}''' \|_{L^\infty}
    \Big)^3
    \sup_{\tilde k\in \mathbb Z}
    \Big\{
        e^{\sigma|\tilde k|^\frac{1}{3}}
        \Big(
        \| \Phi_{{\rm in}, \tilde k} \|_{L^2}
        +
        \|  u_{{\rm in}, \tilde k}       \|_{L^2}
        +
        \| u_{t,{\rm in}, \tilde k} \|_{L^2}
        +
        \| \partial_y u_{{\rm in}, \tilde k} \|_{L^2}
        \Big)
    \Big\}.
\end{aligned}
\end{equation}
Next, we aim to address the sequence of the time derivative $(\partial_t u_k)_{k\in \mathbb Z}$. We invoke once more the relation $u_k = f_k-ik U_{\rm sh}'(\partial_t +1)\psi_k$ in \eqref{eq:fk-and-uk} and we decompose $\partial_t u_k$ as follows:
\begin{align*}
     \partial_t u_k 
     &= \partial_t f_k -i k U_{\rm sh}' (\partial_t +1)\partial_t \psi_k\\
     &= (\partial_t+1) f_k - f_k-i k U_{\rm sh}' \big((\partial_t+i kU_{\rm sh}) (\partial_t +1)-\partial_y^2\big)\psi_k - k^2 U_{\rm sh}' U_{\rm sh}(\partial_t +1) \psi_k
     -ik U_{\rm sh}'\partial_y^2\psi_k.
\end{align*}
Recalling that $ ((\partial_t+i kU_{\rm sh}) (\partial_t +1)-\partial_y^2)\psi_k= \Phi_k$ from  \eqref{eq:psik-not-compact}, we gather that
\begin{equation*}
     \partial_t u_k = (\partial_t+1) f_k - f_k-i k U_{\rm sh}'\Phi_k- k^2 U_{\rm sh}' U_{\rm sh}(\partial_t +1) \psi_k
     -ik U_{\rm sh}'\partial_y^2\psi_k.
\end{equation*}
A straightforward calculation leads hence to the estimate
\begin{align*}
     e^{\gamma(t) |k|^\frac{1}{3}}
     \|  \partial_t u_k  (t) \|_{L^2}
     \leq 
     e^{\gamma(t)|k|^\frac{1}{3}}
     \bigg\{
        &\| (\partial_t +1) f_k (t)   \|_{L^2}+
        \|  f_k  (t) \|_{L^2}+
        |k | \| U_{\rm sh}' \|_{L^\infty}\|  \Phi_k (t) \|_{L^2}+
        \\
        &+
        k^2 \|U_{\rm sh}' \|_{L^\infty }\|U_{\rm sh} \|_{L^\infty }\|  (\partial_t +1) \psi_k (t)  \|_{L^2}+
        |k| \|U_{\rm sh}' \|_{L^\infty } \|\partial_y^2\psi_k(t)     \|_{L^2}
     \bigg\}.
\end{align*}
We now remark that $\gamma(t) \leq \beta(t)$, for any $t\in \mathbb R$. Thus, making use of the Poincar\'e inequalities given by $\|  f_k  (t) \|_{L^2}\leq \| \partial_y f_k  (t) \|_{L^2}$ and $\|  \Phi_k (t) \|_{L^2}\leq \| \partial_y \Phi_k (t) \|_{L^2}= \| u_k \|_{L^2}$, as well as recalling \eqref{ineq:k-est-with-exp} together with
\begin{equation*}
    k^2 =\bigg( \frac{24}{\sigma}\bigg)^6 \bigg( \frac{\sigma}{24} |k|^\frac{1}{3} \bigg)^6
    \leq \bigg( \frac{24}{\sigma}\bigg)^6\Big( e^{\frac{\sigma}{24} |k|^\frac{1}{3}}\Big)^6
    = 
    64\bigg( \frac{12}{\sigma}\bigg)^6 e^{\frac{\sigma}{4} |k|^\frac{1}{3}},
\end{equation*}
we finally obtain
\begin{align*}
     e^{\gamma(t) |k|^\frac{1}{3}}
     &\|  \partial_t u_k  (t) \|_{L^2}
     \leq 
     e^{\beta(t)|k|^\frac{1}{3}}
     \Big(
        \| (\partial_t+1) f_k (t)   \|_{L^2}+
        \| \partial_y f_k  (t) \|_{L^2}
     \Big)+
        \Big(\frac{12}{\sigma}\Big)^3
        \| U_{\rm sh}' \|_{L^\infty}
        e^{\beta(t)|k|^\frac{1}{3}}\|  u_k (t) \|_{L^2}+
        \\&+
        64\bigg( \frac{12}{\sigma}\bigg)^6 \|U_{\rm sh}' \|_{L^\infty }\|U_{\rm sh} \|_{L^\infty }
        e^{\alpha(t)|k|^\frac{1}{3}}
        \|  (\partial_t +1) \partial_y\psi_k (t)  \|_{L^2}+
        \Big(\frac{12}{\sigma}\Big)^3 
        \|U_{\rm sh}' \|_{L^\infty } 
        e^{\alpha(t)|k|^\frac{1}{3}}\|\partial_y^2\psi_k(t)     \|_{L^2}.
\end{align*}
We hence plug \eqref{estimate:main-lemma:dtdypsik+dypsik}, \eqref{ineq:est-on-dt+1fk-dyfk-final} and \eqref{est:ineq-on-uk-dyuk} into this last relation, to gather 
\begin{align*}
     e^{\gamma(t) |k|^\frac{1}{3}}
     \|  \partial_t u_k  (t) \|_{L^2}
     \leq 
     \bigg\{
        \!
        \underbrace{C_{2,\sigma}(t)}_{\text{from all }f_k}
        \!+
        \Big(\frac{12}{\sigma}\Big)^3
        \underbrace{\Big(
         C_{2,\sigma}(t) + \Big(\frac{12}{\sigma}\Big)^3C_\sigma(t)
        \Big)}_{\text{from } u_k}
        \!+ 
        64\bigg( \frac{12}{\sigma}\bigg)^6 
        \hspace{-0.7cm}
        \underbrace{ \frac{C_\sigma(t)}{5}}_{\text{from  }(\partial_t+1) \partial_y\psi_k}
        \hspace{-0.7cm}
        + 
        \Big(\frac{12}{\sigma}\Big)^3
        \!\!
        \underbrace{ \frac{C_\sigma(t)}{5}}_{\text{from  }\partial_y^2\psi_k}
        \!\!
        \bigg\}
     \Big(
        1+ \\
        +
        \| U_{{\rm sh}} \|_{L^\infty}
        +
        \| U_{{\rm sh}}' \|_{L^\infty}
        + 
        \| U_{{\rm sh}}'' \|_{L^\infty}
        + 
        \| U_{{\rm sh}}''' \|_{L^\infty}
    \Big)^4
     \sup_{\tilde k\in \mathbb Z}
     \Big\{
        e^{\sigma|\tilde k|^\frac{1}{3}}
        \Big(
        \|  u_{{\rm in}, \tilde k}       \|_{L^2}
        +
        \| u_{t,{\rm in}, \tilde k} \|_{L^2}
        +
        \| \partial_y u_{{\rm in}, \tilde k} \|_{L^2}
        \Big)
    \Big\}.
\end{align*}
To simplify the summation of the terms depending on $C_\sigma$ and $ C_{2, \sigma}$, we make use of their definitions in \Cref{lemma:Gevrey3-psi} and \eqref{def:C2sigma}, so that
\begin{align*}
     C_{2,\sigma}(t) 
     &+
     \Big(\frac{12}{\sigma}\Big)^3
     \Big(
     C_{2,\sigma}(t) + \Big(\frac{12}{\sigma}\Big)^3C_\sigma(t)
    \Big)
    + 
    64\bigg( \frac{12}{\sigma}\bigg)^6 
    \frac{C_\sigma(t)}{5}   + 
        \Big(\frac{12}{\sigma}\Big)^3
     \frac{C_\sigma(t)}{5}\\
     &\leq 
     \max\bigg\{1, \frac{12}{\sigma}\bigg\}^{6}
     \bigg\{
        2C_{2,\sigma}(t)+14C_{\sigma}(t)
     \bigg\}
     \leq 
     5000
     \max\bigg\{1, \frac{12}{\sigma}\bigg\}^{15}
      (1+t)^5e^t= \frac{\tilde C_\sigma(t)}{2}.
\end{align*}
This provides indeed the following estimate on $\partial_t u_k$
\begin{equation}\label{est:final-estimate-on-dtuk}
\begin{aligned}
     e^{\gamma(t) |k|^\frac{1}{3}}
     \|  \partial_t u_k  (t) \|_{L^2}
     &\leq 
     \frac{\tilde C_\sigma(t)}{2}
     \Big(
       1
       +\| U_{{\rm sh}} \|_{L^\infty}
       +
       \| U_{{\rm sh}}' \|_{L^\infty}
       + \\
       &+ \| U_{{\rm sh}}'' \|_{L^\infty}
       + 
       \| U_{{\rm sh}}''' \|_{L^\infty}
    \Big)^4
      \sup_{\tilde k\in \mathbb Z}
     \Big\{
        e^{\sigma|\tilde k|^\frac{1}{3}}
        \Big(
        \|  u_{{\rm in}, \tilde k}       \|_{L^2}
        +
        \| u_{t,{\rm in}, \tilde k} \|_{L^2}
        +
        \| \partial_y u_{{\rm in}, \tilde k} \|_{L^2}
        \Big)
    \Big\},
\end{aligned}
\end{equation}
which together with \eqref{est:ineq-on-uk-dyuk} (and the fact that $C_{2, \sigma}(t)+(12/\sigma)^3C_\sigma(t) \leq \tilde C(t)/2$) imply finally the claimed inequality \eqref{estimate:main-lemma:uk-Gevrey-3}. This concludes the proof of \Cref{prop:transferring-Gevrey-to-uk}.
\end{proof}

\subsection{A suitable test function}\label{sec:test-function}

In order to conclude the proof of \Cref{main-thm}, we need to establish \Cref{prop:estimate-of-dtdypsik} about the uniform estimate \eqref{main-est:prop-dtdypsik} on $(\partial_t+1)\partial_y \psi_k$ and $ \partial_y^2 \psi_k$.  
In the present section we establish a suitable test function $\omega_{\tau, k}$ for equation \eqref{eq:d_ypsik-final-form}, that reveals some advantageous estimates, when analysing the $L^2$-inner product between $\omega_{\tau, k}$ and equation \eqref{eq:d_ypsik-final-form}. For a given positive time $\tau \in (0,T_\sigma)$, we consider $\omega_{\tau,k}$ as the unique solution of the following backward-in-time linear problem: 
\begin{equation}\label{eq:test-function-omegak}
    \system{
    \begin{alignedat}{4}
    &( (\partial_{t} - 1)(\partial_t  + ik U_{{\rm sh}} )  -\partial_y^2) \omega_{\tau, k} 
    = (\partial_t+1)\partial_y \psi_k 
    \qquad 
    &&(0,\tau) \times (0,1),
    \\
    & \left. \pare{\omega_{\tau, k} , \partial_t \omega_{\tau, k}}\right|_{t=\tau} =\pare{0, 0}
    &&\hspace{1.4cm}(0,1),
    \\
    & \left. \omega_{\tau, k} \right|_{y=0,1}=0
    &&(0,\tau).
    \end{alignedat}
    }
\end{equation}
With the next lemma, we determine the relations between certain meaningful norms of $\omega_{\tau, k}$ and the ones of $(\partial_t+1)\partial_y \psi_k$.
\begin{lemma}\label{lemma:test-function}
    The solution $\omega_{\tau, k}$ of \eqref{eq:test-function-omegak} satisfies  
    at any time $t\in [0,\tau]$
    \begin{equation}\label{est:lemma-omegak}
    \begin{alignedat}{2}
        \sup_{s\in (t, \tau)}
        \| (\partial_t -1)\omega_{\tau, k}(s)\|_{L^2}
        &\leq 
        2\int_t^\tau  \|(\partial_t +1)\partial_y \psi_k(s)\|_{L^2}ds,\\
        \sup_{s\in (t, \tau)}
        \| \partial_y\omega_{\tau, k}(s)\|_{L^2}
        &\leq 
        \sqrt{2}
        \int_t^\tau  \|(\partial_t +1)\partial_y \psi_k(s)\|_{L^2}ds,\\
        \sup_{s\in (t, \tau)}
        \| \omega_{\tau, k}(s)\|_{L^2}
        &\leq 
        2 e^{\tau-t}
        \int_t^\tau
        (s-t)\|(\partial_t +1)\partial_y \psi_k(s)\|_{L^2}ds.
    \end{alignedat}
    \end{equation}
\end{lemma}
\begin{proof} 
We multiply the first equation in \eqref{eq:test-function-omegak} with the complex conjugate $\overline{(\partial_t -1)\omega_{\tau, k}}$. Hence, for a given time $\tilde t\in (0,\tau)$, we integrate the achieved identity within the domain $(\tilde t, \tau)\times (0,1)$ and we extrapolate the corresponding real part:
\begin{equation}\label{est:energy-proof-lemma-omegak}
\begin{aligned}
     -\frac{1}{2}\| (\partial_t -1)\omega_{\tau, k}(\tilde t) \|_{L^2}^2
     -
     \frac{1}{2}
     \|\partial_y\omega_{\tau, k}(\tilde t) \|_{L^2}^2
     &-\int_{\tilde t}^\tau \| \partial_y \omega_{\tau, k}(s) \|_{L^2}^2ds
     \\
     &= 
     \int_{\tilde t}^\tau \int_0^1 
     {\mathfrak{Re}}
     \Big[(\partial_t +1)\partial_y \psi_k \cdot 
     \overline{(\partial_t -1)\omega_{\tau, k}}\Big](s,y)dyds.
\end{aligned}
\end{equation}
We multiply \eqref{est:energy-proof-lemma-omegak} by $-1$ and we take the Supremum of within $\tilde t\in (t,T)$, for a fixed $t\in (0,\tau)$. Thanks to  Cauchy-Schwarz, we hence establish that
\begin{align*}
    \sup_{\tilde t\in (t, \tau)}
    &\| (\partial_t -1)\omega_{\tau, k}(\tilde t) \|_{L^2}^2
    \leq 
    2
    \int_t^\tau
        \| (\partial_t +1)\partial_y \psi_k(s) \|_{L^2}
        \| (\partial_t -1) \omega_{\tau, k}(s) \|_{L^2}
    ds\\
    &\leq 
    2
    \int_t^\tau
        \| (\partial_t +1)\partial_y \psi_k(s) \|_{L^2}
    ds
    \sup_{s\in (t,\tau)}
    \| (\partial_t -1) \omega_{\tau, k}(s) \|_{L^2}.
\end{align*}
This corresponds to the first inequality of \eqref{est:lemma-omegak}.
Next, we deal with the norm $ \| \partial_y \omega_{\tau, k} \|_{L^2}$ in \eqref{est:lemma-omegak}. By invoking \eqref{est:energy-proof-lemma-omegak}, we have first 
\begin{align*}
     \sup_{s\in (t, T)}\| \partial_y \omega_{\tau, k} (s) \|_{L^2} 
     &\leq
     \bigg(
        \int_t^\tau
        \| (\partial_t +1)\partial_y \psi_k(s) \|_{L^2}
        ds
        \sup_{s\in (t, T)}
        \| (\partial_t -1) \omega_{\tau, k}(s) \|_{L^2}
    \bigg)^\frac{1}{2}.
\end{align*}
The result is thus obtained by invoking the first inequality of \eqref{est:lemma-omegak}.
To conclude the proof, we deal now with the last inequality of \eqref{est:lemma-omegak}. Since $\omega_{\tau, k}$ is null at $t = \tau$, we have
\begin{align*}
    \| \omega_{\tau, k}(t) \|_{L^2}
    &= 
    \Big\| -\int_t^\tau
    \partial_t\omega_{\tau, k}(s)ds \Big\|_{L^2}
    \leq 
    \int_t^\tau
    \| \partial_t \omega_{\tau, k}(s) \|_{L^2}ds\\
    &\leq 
    \int_t^\tau
    \| (\partial_t-1) \omega_{\tau, k}(s) \|_{L^2}ds+
    \int_t^\tau
    \| \omega_{\tau, k}(s) \|_{L^2}ds.
\end{align*}
Furthermore, the first estimate in \eqref{est:lemma-omegak} guarantees that
\begin{equation*}
    \| \omega_{\tau, k}(t) \|_{L^2}
    \leq 
    2
    \int_t^\tau
    \int_s^\tau
    \| (\partial_t +1)\partial_y \psi_k(z) \|_{L^2}
    dz
    ds +
    \int_t^\tau
    \| \omega_{\tau, k}(s) \|_{L^2}ds.
\end{equation*}
The result is then achieved by applying the Gronwall's lemma:
\begin{equation*}
    \| \omega_{\tau, k}(t) \|_{L^2} 
    \leq 
    2 e^{\tau-t}
    \int_t^\tau
    \int_s^\tau
    \| (\partial_t +1)\partial_y \psi_k(z) \|_{L^2}
    ds.
\end{equation*}
This concludes the proof of \Cref{lemma:test-function}.
\end{proof}

\subsection{Proof of \Cref{prop:estimate-of-dtdypsik}}\label{sec:the-main-estimate}

\noindent 
This section is devoted to the proof of \Cref{prop:estimate-of-dtdypsik}, which is based on the specific test function $\overline{\omega_{\tau, k}}$, introduced in \Cref{sec:test-function}. 
We begin with, by recalling system \eqref{eq:d_ypsik-final-form} for the evolution of $\partial_y \psi_k$:
\begin{equation}\label{eq:d_ypsik-final-form-sec:main-estimate}
\begin{aligned}
     \big( (\partial_{t} + 1)(\partial_t  + ik U_{{\rm sh}} )  -\partial_y^2\big)^2
    \partial_y
    \psi_k
    = 
   [ik U_{{\rm sh}}'  , \partial_{y}^2]\big( (\partial_t + 1)\psi_k\big),
\end{aligned}
\end{equation}
with initial data $\partial_t \psi_{k|t = 0} =\psi_{k|t = 0} = 0 $ and boundary conditions $\psi_{k|y = 0,1} = 0$. Next, we fix a general time $\tau \in (0,T_\sigma)$ and we multiply the equation \eqref{eq:d_ypsik-final-form-sec:main-estimate} with the conjugate $\overline{\omega_{\tau, k}}$ of the test function defined in \eqref{eq:test-function-omegak}. By integrating the result along $(0, \tau)\times (0,1)$, we obtain the following identity:
\begin{equation}\label{identity1-sec-main-estimate}
\begin{aligned}
    \int_0^\tau \int_0^1 
    \Big[
    \big( (\partial_{t} + 1)(\partial_t  &+ ik U_{{\rm sh}}(y) )  -\partial_y^2\big)^2
        \partial_y
        \psi_k
    \Big](t,y)
    \overline{\omega_{\tau, k}(t,y)}
    dy dt\\
    &= 
     \int_0^\tau \int_0^1 
   [ik U_{{\rm sh}}'(y) , \partial_{y}^2]\big( (\partial_t + 1)\psi_k(t,y)\big)
   \overline{\omega_{\tau, k}(t,y)}
    dy dt.
\end{aligned}
\end{equation}
We aim therefore to integrate by parts the integral at the left-hand side.
To this end, we first develop the operator $(\partial_{t} + 1)(\partial_t  + ik U_{{\rm sh}}(y) ) -\partial_y^2$ into $\partial_{t}^2 + (1+ik U_{{\rm sh}}(y))\partial_t + ik U_{{\rm sh}}(y) - \partial_y^2$, which localises the order of each derivative. Hence, by considering momentarily the derivative $\partial_t^2$ of second order, we gather
\begin{equation}\label{integration-by-parts-dtt1}
\begin{aligned}
     \int_0^\tau \int_0^1 
    \Big[
    \partial_t^2 
        \big( (\partial_{t} &+ 1)(\partial_t  + ik U_{{\rm sh}}(y) )  -\partial_y^2\big)
        \partial_y
        \psi_k
    \Big](t,y)
    \overline{\omega_{\tau, k}(t,y)}
    dy dt\\
    =
    &\int_0^\tau \int_0^1 
    \Big[
        \big( (\partial_{t}+ 1)(\partial_t  + ik U_{{\rm sh}}(y) )  -\partial_y^2\big)
        \partial_y
        \psi_k
    \Big](t,y)
    \overline{\partial_t^2 \omega_{\tau, k}(t,y)}
    dy dt  
    + 
    \\
    &
    +
    \int_0^1 
    \Big[
        \partial_t \big( (\partial_{t}+ 1)(\partial_t  + ik U_{{\rm sh}}(y) )  -\partial_y^2\big)
        \partial_y
        \psi_k
    \Big](0,y)
    \overline{ \omega_{\tau, k}(0,y)}
    dy + \\
    &-
    \int_0^1 
    \Big[
        \big( (\partial_{t}+ 1)(\partial_t  + ik U_{{\rm sh}}(y) )  -\partial_y^2\big)
        \partial_y
        \psi_k
    \Big](0,y)
    \overline{\partial_t \omega_{\tau, k}(0,y)}
    dy.
\end{aligned}
\end{equation}
The last two integrals of \eqref{integration-by-parts-dtt1} are set at $t= 0$ and can hence be recasted in terms of the initial data of the velocity field $u_{{\rm in},k}$, $\partial_t u_{{\rm in},k}$ and of the stream function $\Phi_{{\rm in},k}$.  Indeed, recalling that $\psi_k$ is also solution of \eqref{eq:psik-not-compact}, we remark that the second integrand at the right-hand side of \eqref{integration-by-parts-dtt1} satisfies
\begin{equation*}
    \Big[
    \partial_t( (\partial_{t} + 1)(\partial_t  + ik U_{{\rm sh}} )  -\partial_y^2) \partial_y \psi_k
    \Big](0,y)
    = ik U_{{\rm sh}}'(y) 
    \big(\partial_t^2 \psi_k(0,y)+ \partial_t \psi_k (0,y)\big) + \partial_t \partial_y \Phi_{k}(0,y),
\end{equation*}
for any $y\in (0,1)$. This can be simplified further, since $ \partial_t \partial_y \Phi_{k}(0,y)=\partial_y \Phi_{t, {\rm in}, k} = u_{t,{\rm in}, k}$,  $\partial_t \psi_k|_{t= 0} = 0$ and equation \eqref{eq:psik-not-compact} implies that $\partial_t^2\psi_k|_{t= 0} = \Phi_{\rm in,k}$. Thus
\begin{equation}\label{integration-by-parts-initial-time-1}
    \Big[
    \partial_t( (\partial_{t} + 1)(\partial_t  + ik U_{{\rm sh}} )  -\partial_y^2) \partial_y \psi_k
    \Big](0,y)
    = ik U_{{\rm sh}}'(y)  \Phi_{{\rm in},k}(y)
     +  u_{t,{\rm in}, k}(y),
\end{equation}
for any $y\in (0,1)$. An analogous approach leads moreover to the following identity for the third integrand at the right-hand side of \eqref{integration-by-parts-dtt1}:
\begin{equation}\label{integration-by-parts-initial-time-2}
    \Big[
    ( (\partial_{t} + 1)(\partial_t  + ik U_{{\rm sh}} )  -\partial_y^2) \partial_y \psi_k
    \Big](0,y)=
    \partial_y \Phi_{k}(0,y)= u_{{\rm in}, k}(y),\qquad 
    y\in (0,1).
\end{equation}
Therefore, thanks to the relations \eqref{integration-by-parts-initial-time-1} and \eqref{integration-by-parts-initial-time-2}, we can reformulate \eqref{integration-by-parts-dtt1} as follows:
\begin{equation}\label{integration-by-parts-dtt2}
\begin{aligned}
     \int_0^\tau \int_0^1 
    \Big[
    \partial_t^2 
        \big( (\partial_{t} &+ 1)(\partial_t  + ik U_{{\rm sh}}(y) )  -\partial_y^2\big)
        \partial_y
        \psi_k
    \Big](t,y)
    \overline{\omega_{\tau, k}(t,y)}
    dy dt\\
    &=
    \int_0^\tau \int_0^1 
    \Big[
        \big( (\partial_{t}+ 1)(\partial_t  + ik U_{{\rm sh}}(y) )  -\partial_y^2\big)
        \partial_y
        \psi_k
    \Big](t,y)
    \overline{\partial_t^2 \omega_{\tau, k}(t,y)}
    dy dt + 
    \\
    &
    +
    \int_0^1 
    \Big(
        ik U_{{\rm sh}}'(y)  \Phi_{{\rm in},k}(y)
     +  u_{t,{\rm in}, k}(y)
    \Big)
    \overline{ \omega_{\tau, k}(0,y)}
    dy -
    \int_0^1 
    u_{{\rm in}, k}(y)
    \overline{\partial_t \omega_{\tau, k}(0,y)}
    dy.
\end{aligned}
\end{equation}
We now come back to our original identity \eqref{identity1-sec-main-estimate} and  we shall now integrate by parts the operator $(1+ikU_{{\rm sh}})\partial_t$, with a a first order derivative. As for \eqref{integration-by-parts-dtt2}, our aim is once more to recast the resulting integrals at $t= 0$ in terms of the initial data. A direct calculation guarantees that
\begin{equation}
\begin{aligned}
    \int_0^\tau \int_0^1 
    \Big[
    (1&+ik U_{{\rm sh}} )
    \partial_t 
        \big( (\partial_{t} + 1)(\partial_t  + ik U_{{\rm sh}} )  -\partial_y^2\big)
        \partial_y
        \psi_k
    \Big](t,y)
    \overline{\omega_{\tau, k}(t,y)}
    dy dt\\
    =
    &\int_0^\tau \int_0^1 
    \Big[
        \big( (\partial_{t}+ 1)(\partial_t  + ik U_{{\rm sh}}  )  -\partial_y^2\big)
        \partial_y
        \psi_k
    \Big](t,y)
    \overline{(1-ikU_{{\rm sh}}(y))\partial_t \omega_{\tau, k}(t,y)}
    dy dt+ \\
    &+
    \int_0^1 
    \Big[
        (1+ik U_{{\rm sh}} )
        \big( (\partial_{t}+ 1)(\partial_t  + ik U_{{\rm sh}}  )  -\partial_y^2\big)
        \partial_y
        \psi_k
    \Big](0,y)
    \overline{\omega_{\tau, k}(0,y)}
    dy.
\end{aligned}
\end{equation}
Hence, recalling from \eqref{integration-by-parts-initial-time-2} that 
$ [((1+ik U_{{\rm sh}} )( (\partial_{t}+ 1)(\partial_t  + ik U_{{\rm sh}}  )  -\partial_y^2) \partial_y\psi_k](0,y)= u_{{\rm in}, k}(y)$, we obtain
\begin{equation}\label{integration-by-parts-dt}
\begin{aligned}
     \int_0^\tau \int_0^1
    \Big[
    (1 &+ik U_{{\rm sh}} )
     \partial_t 
        \big( (\partial_{t} + 1)(\partial_t  +  ik U_{{\rm sh}} )  \!-\!\partial_y^2\big)
        \partial_y
        \psi_k
    \Big](t,y)
    \overline{\omega_{\tau, k}(t, y)}
    dy dt
    \\
    =
    &\int_0^\tau \int_0^1 
    \Big[
        \big( (\partial_{t}+ 1)(\partial_t  + ik U_{{\rm sh}}(y) )  -\partial_y^2\big)
        \partial_y
        \psi_k
    \Big](t,y)
    \overline{(1-ikU_{{\rm sh}}(y))\partial_t \omega_{\tau, k}(t,y)}
    dy dt +\\
    &+\int_0^1 
    (1 +  ikU_{{\rm sh}}(y)\!)u_{{\rm in}, k}(y)
    \overline{\omega_{\tau, k}(0, y)}
    dy.
\end{aligned}
\end{equation}
To conclude the integration by parts related to the operator $((\partial_{t} + 1)(\partial_t  + ik U_{{\rm sh}}(y) )  -\partial_y^2)$ in \eqref{identity1-sec-main-estimate}, we shall now treat $-\partial_y^2$. Making use of the homogeneous conditions $\omega_{k|y = 0,1} = 0$ on the test function, we have that
\begin{align*}
    -\int_0^\tau \int_0^1 
    &\Big[
    \partial_y^2 
        \big( (\partial_{t} + 1)(\partial_t  + ik U_{{\rm sh}}(y) )  -\partial_y^2\big)
        \partial_y
        \psi_k
    \Big](t,y)
    \overline{\omega_{\tau, k}(t,y)}
    dy dt \\
    &=\int_0^\tau \int_0^1 
    \Big[
        \partial_y\big( (\partial_{t} + 1)(\partial_t  + ik U_{{\rm sh}}(y) )  -\partial_y^2\big)
        \partial_y
        \psi_k
    \Big](t,y)
    \overline{\partial_y \omega_{\tau, k}(t,y)}
    dy dt.
\end{align*}
Now, recalling from \eqref{eq:fk-and-uk} that $ ( (\partial_{t} + 1)(\partial_t  + ik U_{{\rm sh}}(y) )  -\partial_y^2 )       \partial_y        \psi_k = f_k = i k U_{\rm sh}'(y) (\partial_t +1) \psi_k + u_k $ is null in $y =0,1$, we obtain
\begin{equation}\label{no-idea-what-i-am-doing}
    \begin{aligned}
    -\int_0^\tau \int_0^1 
    &\Big[
    \partial_y^2 
        \big( (\partial_{t} + 1)(\partial_t  + ik U_{{\rm sh}}(y) )  -\partial_y^2\big)
        \partial_y
        \psi_k
    \Big](t,y)
    \overline{\omega_{\tau, k}(t,y)}
    dy dt \\
    &=
    -\int_0^\tau \int_0^1 
    \Big[
        \big( (\partial_{t} + 1)(\partial_t  + ik U_{{\rm sh}}(y) )  -\partial_y^2\big)
        \partial_y
        \psi_k
    \Big](t,y)
    \overline{\partial_y^2 \omega_{\tau, k}(t,y)}
    dy dt.
\end{aligned}
\end{equation}
As final result, we couple the identities 
\eqref{integration-by-parts-dtt2},\eqref{integration-by-parts-dt}  and \eqref{no-idea-what-i-am-doing}, so that \eqref{identity1-sec-main-estimate} can be recasted as
\begin{equation}\label{eq:dypsk-omegak}
\begin{aligned}
    \int_0^\tau \int_0^1 
    \big( (\partial_{t} + 1)(\partial_t  + ik U_{{\rm sh}}(y) )  -\partial_y^2\big)
        \partial_y
        \psi_k(t,y)
    \overline{
    \big( (\partial_{t}- 1)(\partial_t  + ik U_{{\rm sh}}(y) )  -\partial_y^2\big)\omega_{\tau, k}(t,y)
    }
    dy dt + \\ 
    +
    \int_0^1 
    \Big(
        ik U_{{\rm sh}}'(y)  \Phi_{{\rm in},k}(y)
     +  u_{t,{\rm in}, k}(y) +
      (1+ikU_{{\rm sh}}(y))u_{{\rm in}, k}(y)
    \Big)
    \overline{ \omega_{\tau, k}(0,y)}
    dy - \\ 
    -
    \int_0^1 
    u_{{\rm in}, k}(y)
    \overline{\partial_t \omega_{\tau, k}(0,y)}
    dy 
    = 
     \int_0^\tau \int_0^1 
   [ik U_{{\rm sh}}'(y) , \partial_{y}^2]\big( (\partial_t + 1)\psi_k(t,y)\big)
   \overline{\omega_{\tau, k}(t,y)}
    dy dt.
\end{aligned}
\end{equation}
Next, we make use of \eqref{eq:dypsk-omegak}, in order to derive suitable estimates on the $L^2$-norms of $(\partial_t +1) \partial_y \psi_k $ and $\partial_y^2 \psi_k$. These estimates shall not depend upon  $\omega_{\tau, k}$, hence we aim to get rid of this test function making use of \Cref{lemma:test-function} and system \eqref{eq:test-function-omegak}. By extrapolating the real part of \eqref{eq:dypsk-omegak}, the first integral becomes
\begin{align*}
    \mathfrak{Re}
    \int_0^\tau\!\!\!\int_0^1 
    \big( (\partial_{t} &+ 1)(\partial_t  + ik U_{{\rm sh}}(y) )  -\partial_y^2\big)
        \partial_y
        \psi_k(t,y)
    \overline{
    (\partial_t +1)\partial_y\psi_k(t,y)
    }dy dt\\
    &=
    \int_0^\tau
    \int_0^1 
    \mathfrak{Re}
    \big[
    \partial_t 
    (\partial_{t} + 1)\partial_y\psi_k(t,y)
    \overline{
    (\partial_t +1)\partial_y\psi_k(t,y)
    }
    \,\big]
    dy dt
    + \\ 
    &
    \qquad +
    \int_0^\tau
    \int_0^1 
    \underbrace{
    \mathfrak{Re}
    \big[
    i k U_{\rm sh}(y)
    \| (\partial_{t} + 1)\partial_y\psi_k(t,y)\|_{L^2}^2
    \big]
    }_{= 0}
    dy dt
    +
    \int_0^\tau  
    \big\| \partial_y^2 \psi_k(s) \big\|_{L^2}^2ds
    \\
    &=
    \frac{1}{2}
    \big\| (\partial_t +1)\partial_y \psi_k(\tau) \big\|_{L^2}^2
    +
    \frac{1}{2}
    \big\| \partial_y \psi_k(\tau) \big\|_{L^2}^2
    +
    \int_0^\tau  
    \big\| \partial_y^2 \psi_k(s) \big\|_{L^2}^2ds,
\end{align*}
where we have used in the last identity the fact that $ \partial_y \psi_k$ and $ \partial_t \partial_y \psi_k$ are identically null at $t = 0$. Furthermore, when dealing with the real part of the second and third integrals in \eqref{eq:dypsk-omegak}, we obtain
\begin{equation}\label{est:real-part-robaccia1}
\begin{aligned}
    \bigg|
    \mathfrak{Re}
    &\int_0^1 
    \Big(
        ik U_{{\rm sh}}'(y)  \Phi_{{\rm in},k}(y)
     +  u_{t,{\rm in}, k}(y) +
      (1+ikU_{{\rm sh}}(y))u_{{\rm in}, k}(y)
    \Big)
    \overline{ \omega_{\tau, k}(0,y)}
    dy 
    \bigg|
    \\
    &\leq
    \Big(
    |k | \| U_{{\rm sh}}' \|_{L^\infty} \| \Phi_{\rm in,k} \|_{L^2} + 
    \|  u_{t,{\rm in}, k} \|_{L^2} + 
    \big(1+ |k|  \| U_{{\rm sh}}  \|_{L^\infty} \big)\|  u_{{\rm in}, k} \|_{L^2} 
    \Big) \| \omega_{\tau, k}(0) \|_{L^2},
\end{aligned}
\end{equation}
as well as
\begin{equation}\label{est:real-part-robaccia2}
\begin{aligned}
    \bigg|
    \mathfrak{Re}
    &\int_0^1 
    u_{{\rm in}, k}(y)
    \overline{\partial_t \omega_{\tau, k}(0,y)}
    dy  
    \bigg|
    \leq
    \|  u_{ {\rm in}, k} \|_{L^2}  
    \| \partial_t \omega_{\tau, k}(0) \|_{L^2}.
\end{aligned}
\end{equation}
Finally, the real part of the right-hand side in \eqref{eq:dypsk-omegak} fulfills  
\begin{equation}\label{est:real-part-robaccia3}
\begin{aligned}
   \bigg|
   \mathfrak{Re}
   \int_0^\tau \int_0^1 
   [ik U_{{\rm sh}}'(y) &, \partial_{y}^2]\big( (\partial_t + 1)\psi_k(t,y)\big)
   \overline{\omega_{\tau, k}(t,y)}
    dy dt 
    \bigg|\\
    &\leq
    |k|\int_0^\tau \!
        \|
            [ U_{{\rm sh}}', \partial_{y}^2] 
            (\partial_t +  1)\psi_k(t)
        \|_{L^2}
        \|  
            \omega_{\tau, k}(t)
        \|_{L^2}
        dt\\
        &\leq 
         |k|\int_0^\tau\|
            [ U_{{\rm sh}}', \partial_{y}^2] 
            (\partial_t +  1)\psi_k(t)
        \|_{L^2}
        \|  
            \omega_{\tau, k}(t)
        \|_{L^2}
        dt\\
        &\leq 
         |k|\int_0^\tau\|
            U_{\rm sh}'''
            (\partial_t +  1)\psi_k(t)
            +
            2U_{\rm sh}''
            (\partial_t +  1)\partial_y \psi_k(t)
        \|_{L^2}
        \|  
            \omega_{\tau, k}(t)
        \|_{L^2}
        dt
        \\
        &\leq 
         |k|\int_0^\tau
        \Big(
            \| U_{\rm sh}'''\|_{L^\infty}
            +
            2
            \| U_{\rm sh}'' \|_{L^\infty}
        \Big)
        \| (\partial_t +  1)\partial_y \psi_k(t) \|_{L^2}
        \|  
            \omega_{\tau, k}(t)
        \|_{L^2}
        dt,
\end{aligned}
\end{equation}
where we have also made use of the Poincar\'e inequality in $y\in (0,1)$: $   \| (\partial_t \!+\!  1) \psi_k \|_{L^2}\leq  \| (\partial_t \!+\!  1)\partial_y \psi_k  \|_{L^2}$.
We can summarise hence our last estimates, by coupling \eqref{eq:dypsk-omegak} together with \eqref{est:real-part-robaccia1}, \eqref{est:real-part-robaccia2} and \eqref{est:real-part-robaccia3}. This guarantees that for any $\tau \in (0, T_\sigma)$
\begin{equation}\label{est:first-estimate-for-||(dt+1)dypsi_k||^2}
\begin{aligned}
    \frac{1}{2}
    \big\| &(\partial_t +1)\partial_y \psi_k(\tau ) \big\|_{L^2}^2
    +
    \frac{1}{2}
    \big\| \partial_y^2 \psi_k(\tau ) \big\|_{L^2}^2
    +
    \int_0^\tau  
    \big\| \partial_y^2 \psi_k(t) \big\|_{L^2}^2dt
    \\
    \leq 
    &    \Big(
    |k | \| U_{{\rm sh}}' \|_{L^\infty} \| \Phi_{\rm in,k} \|_{L^2} + 
    \|  u_{t,{\rm in}, k} \|_{L^2} + 
    \big(1+ |k|  \| U_{{\rm sh}}  \|_{L^\infty} \big)\|  u_{{\rm in}, k} \|_{L^2} 
    \Big) \| \omega_{\tau, k}(0) \|_{L^2} \\
    &+
        \| u_{{\rm in}, k}                  \|_{L^2}
        \| \partial_t \omega_{\tau, k}(0)   \|_{L^2}
     + 
        |k|
        \Big(
            \| U_{\rm sh}'''\|_{L^\infty}
            +
            2
            \| U_{\rm sh}'' \|_{L^\infty}
        \Big)
        \int_0^\tau
        \| (\partial_t +  1)\partial_y \psi_k(t) \|_{L^2}
        \|  
            \omega_{\tau, k}(t)
        \|_{L^2}
        dt.
\end{aligned}
\end{equation}
The right-hand side still depends upon the test function $\omega_{k, \tau}$. We are however in the condition to get rid of that, by applying \Cref{lemma:test-function}. This implies in particular (together with Poincar\'e) that 
\begin{align*}
    \|\omega_{k, \tau}(0) \|_{L^2}
    &\leq 
    \|\partial_y \omega_{k, \tau}(0) \|_{L^2}
    \leq \sqrt{2} \int_0^\tau   \| (\partial_t +1) \partial_y \psi_k(s) \|_{L^2}ds\\
    &\leq 
    \sqrt{2}\,\tau \!\!\sup_{s\in [0,\tau] }  
    \| (\partial_t  + 1)\partial_y \psi_k(s ) \|_{L^2},\\
    \|\partial_t \omega_{k, \tau}(0) \|_{L^2}
    &\leq 
    \|(\partial_t-1) \omega_{k, \tau}(0) \|_{L^2}+
    \|  \omega_{k, \tau}(0) \|_{L^2}\leq  
    (2+\sqrt{2})
    \int_0^\tau
    \| (\partial_t + 1)\partial_y\psi_{k}(s)   \|_{L^2}
    ds\\
    &\leq 
    \sqrt{2}(1+\sqrt{2})\tau \!\!\sup_{s\in [0,\tau] }  
    \| (\partial_t  + 1)\partial_y \psi_k(s ) \|_{L^2}
    ,\\
    \| \omega_k(t) \|_{L^2}
    &\leq 
    e^{\tau-t}
     \int_t^\tau
        (s-t)\|(\partial_t +1)\partial_y \psi_k(s)\|_{L^2}ds\\
     &\leq 
    e^{\tau}
     \int_t^\tau
    (s-t)ds
    \!\!\sup_{s\in [0,\tau] }  
    \| (\partial_t  + 1)\partial_y \psi_k(s ) \|_{L^2} 
    = 
    \frac{e^{\tau}
    (\tau-t)^2}{2}
    \!\!\sup_{s\in [0,\tau] }  
    \| (\partial_t  + 1)\partial_y \psi_k(s ) \|_{L^2} .
\end{align*}
Thus, by taking now the supremum in $\tau \in (0,t)$ for a general $t\in (0,T_\sigma)$ (and re-denoting the variables of integration), we can recast \eqref{est:first-estimate-for-||(dt+1)dypsi_k||^2} uniquely in terms of the stream function $\psi_k$ as follows:
\begin{equation*}
\begin{aligned}
    &\frac{1}{2}
    \sup_{s \in (0,t)}
    \bigg\{
        \big\| (\partial_t +1)\partial_y \psi_k(s) \big\|_{L^2}^2
        +
        \big\| \partial_y^2 \psi_k(s) \big\|_{L^2}^2
    \bigg\}
    +
    \int_0^t  
    \big\| \partial_y^2 \psi_k(s) \big\|_{L^2}^2ds
    \leq  
     \\ 
    &
    \sup_{s\in (0,t) } \!\!
    \| (\partial_t \!+\!1)\partial_y \psi_k(s ) \|_{L^2}
    \bigg\{
        \sqrt{2}
        \Big(
        |k | \| U_{{\rm sh}}' \|_{L^\infty} \| \Phi_{\rm in,k} \|_{L^2} + 
        \|  u_{t,{\rm in}, k} \|_{L^2} + 
        \big(1+ |k|  \| U_{{\rm sh}}  \|_{L^\infty} \big)\|  u_{{\rm in}, k} \|_{L^2} 
        \Big)t
        +\\
    & \hspace{2cm}  
        +2(1+\sqrt{2})\| u_{\rm in} \|_{L^2}
        \tau
        + 
        |k|
        \Big(
            \| U_{\rm sh}'''\|_{L^\infty}
            +
            2
            \| U_{\rm sh}'' \|_{L^\infty}
        \Big)
        e^{t}
        \int_0^t
        \| (\partial_t +  1)\partial_y \psi_k(s) \|_{L^2}
        (t-s)^2
        ds
    \bigg\}.
\end{aligned}
\end{equation*}
If $\sup_{s\in (0,t) } \| (\partial_t \!+\!1)\partial_y \psi_k(s ) \|_{L^2} = 0$, then the main estimate \eqref{main-est:prop-dtdypsik} is automatically satisfied. On the other hand, in case this term is not identically null, we have that
\begin{align*}
    &
    \sup_{s \in (0,t)}
    \bigg\{
        \big\| (\partial_t +1)\partial_y \psi_k(s) \big\|_{L^2}
        +
        \big\| \partial_y^2 \psi_k(s) \big\|_{L^2}
    \bigg\}
    \leq 
    \sqrt{2}
    \bigg(
    \sup_{s \in (0,t)}
    \bigg\{
        \big\| (\partial_t +1)\partial_y \psi_k(s ) \big\|_{L^2}^2
        +
        \big\| \partial_y^2 \psi_k(s) \big\|_{L^2}^2 
    \bigg\}
    \bigg)^\frac{1}{2}\\
    &\leq 
    \frac{2\sqrt{2}}{\sup_{s\in (0,t) } \| (\partial_t \!+\!1)\partial_y \psi_k(s ) \|_{L^2}}
    \frac{1}{2}
    \sup_{s \in (0,t)}
    \bigg\{
        \big\| (\partial_t +1)\partial_y \psi_k(s ) \big\|_{L^2}^2
        +
        \big\| \partial_y^2 \psi_k(s) \big\|_{L^2}^2 
    \bigg\}\\
    &\leq 
    2\sqrt{2}
    \bigg\{
        \sqrt{2}
        \Big(
        |k | \| U_{{\rm sh}}' \|_{L^\infty} \| \Phi_{\rm in,k} \|_{L^2} + 
        \|  u_{t,{\rm in}, k} \|_{L^2} + 
        \big(1+ |k|  \| U_{{\rm sh}}  \|_{L^\infty} \big)\|  u_{{\rm in}, k} \|_{L^2} 
        \Big)t
        +\\
    & \hspace{2cm}  
        +2(1+\sqrt{2})\| u_{{\rm in}, k} \|_{L^2}t
        + 
        |k|
        \Big(
            \| U_{\rm sh}'''\|_{L^\infty}
            +
            2
            \| U_{\rm sh}'' \|_{L^\infty}
        \Big)
        e^{t}
        \int_0^t
        \| (\partial_t +  1)\partial_y \psi_k(s) \|_{L^2}
        (t-s)^2
        ds
    \bigg\}.
\end{align*}
We hence reorganise the last inequality into the following compact form:
\begin{equation*}
    \sup_{s \in [0,t]}
    \bigg\{
        \big\| (\partial_t +1)\partial_y \psi_k(s) \big\|_{L^2}
        +
        \big\| \partial_y^2 \psi_k(s) \big\|_{L^2}
    \bigg\}
    \leq g_k(t) + \frac{\lambda_k(t)^3}{2}
    \int_0^t
    (t-s)^2
    \sup_{\tau \in [0,s]}\| (\partial_t +  1)\partial_y \psi_k(\tau) \|_{L^2}
    ds
\end{equation*}
where the functions $g_k(t)$ and $\lambda_k(t)$ are defined by means of
\begin{align*}
    g_k(t) 
    &:= 
    4t 
    \Big\{
        |k|
        \Big(
            \|U_{\rm sh}' \|_{L^\infty}\| \Phi_{\rm in, k}  \|_{L^2} + 
            \|U_{\rm sh}  \|_{L^\infty}\| u_{\rm in, k}     \|_{L^2}
        \Big)
        +
        \| u_{t,\rm in, k}     \|_{L^2}
        +
        (3+\sqrt{2})
        \| u_{\rm in, k}     \|_{L^2}
    \Big\}
    \\
    \lambda_k(t)
    &:= 2^{\frac{5}{6}}|k|^\frac{1}{3}
    \Big(
            \| U_{\rm sh}'''\|_{L^\infty}
            +
            2
            \| U_{\rm sh}'' \|_{L^\infty}
    \Big)^\frac{1}{3}
    e^{\frac{t}{3}}.
\end{align*}
This last inequality corresponds to our claimed estimate \eqref{main-est:prop-dtdypsik}. This concludes therefore the proof of \Cref{prop:estimate-of-dtdypsik}.

\section{Conclusion and remarks on the non-linear system}\label{sec:remarks-on-the-nonlinear-system}
In this section, we investigate why Theorem \ref{main-thm} cannot be proven for the nonlinear counterparts of \eqref{LHP} without further ado. Clearly, this is a consequence of the nonlinear structure but furthermore, the hyperbolic regime interferes with the known cancellation properties of  the classical Prandtl/Navier-Stokes equations in anaggravating way.

\smallskip
\noindent
To begin with, we observe that several candidates exist for which Theorem \ref{main-thm} might hold true. Some of the represent simplifications of other formulations but, nevertheless, they contain drawbacks which cannot be dealt with easily. 
The simplest form of the nonlinear hyperbolic Prandtl equation consists of 
\begin{equation} \label{simpleHP}
    \system{
    \begin{alignedat}{4}
    & 
        \tau \partial_{tt} u  + \partial_t u +   u  \partial_x u 
    + v \partial_y u  
    - \partial_{y}^2 u =
     \big( 
        \tau \partial_t+1
    \big)
    \big(
        \partial_t u^{E}+ u^E\partial_x u^E
    \big),
    \quad  
    &&(0,T) \times \mathbb{X}\times (0,+\infty),\\
    & \partial_x u + \partial_y v =0
    &&(0,T) \times \mathbb{X}\times (0,+\infty),
    \\
    & \left. \pare{u,v}\right|_{y=0}=0\quad 
    \lim_{y\to+\infty} u = u^E
     &&(0,T) \times \mathbb{X},\\
    &\left. \pare{u, u_t}\right|_{t=0} =\pare{u_{\rm in}, u_{t,\rm in}}
    &&\hspace{1.4cm} \mathbb{X}\times (0,+\infty),
    \end{alignedat}
    },
\end{equation}
This system looks promising when trying to implement the strategy of \cite{MR3925144}. However, besides the fact that the second time derivative produces difficulties (see below), a quick look at the linearization
\begin{align*}
    & \partial_{t}^2 u +
    \partial_tu  + U_{{\rm sh}}(y) \partial_x u + v \ U'_{{\rm sh}}(y) - \partial_{yy} u =0,
    \qquad && \text{on } (0,T) \times \mathbb{T}\times (0,1)
\end{align*}
suffices to realize the Eigenvalues contain a positive real part in general. More precisely, solutions corresponding to a frequency $k$ in $x$ will behave like $e^{\sqrt{|k|}t}$ which restricts well-posedness theory to the Gevrey 2 case. For a hyperbolic equation, this is expected and actually proven for  \eqref{simpleHP} in \cite{LiXu2021} (see also \cite{PZ2022}).

\smallskip
\noindent
In conclusion, it is essential to maintain the convective structure of the hyperbolic Prandtl equations (as shown in Section \ref{sec:two}). By Cattaneo's law, it reads 
\begin{equation}
    \system{
    \begin{alignedat}{4}
    &  
    \big( 
        \tau \partial_t+1
    \big)
    \big(
    \partial_t u  +  u  \partial_x u 
    + v \partial_y u  
    \big)
    - \partial_{y}^2 u =
     \big( 
        \tau \partial_t+1
    \big)
    \big(
        \partial_t u^{E}+ u^E\partial_x u^E
    \big),
    \quad  
    &&(0,T) \times \mathbb{X}\times (0,+\infty),\\
    & \partial_x u + \partial_y v =0
    &&(0,T) \times \mathbb{X}\times (0,+\infty),
    \\
    & \left. \pare{u,v}\right|_{y=0}=0\quad 
    \lim_{y\to+\infty} u = u^E
     &&(0,T) \times \mathbb{X},\\
    &\left. \pare{u, u_t}\right|_{t=0} =\pare{u_{\rm in}, u_{t,\rm in}}
    &&\hspace{1.4cm} \mathbb{X}\times (0,+\infty),
    \end{alignedat}
    }.
\end{equation}
Unfortunately, the time derivative on the convective term brings several other difficulties with it. At first, note that one of the terms, $\partial_t u \partial_x u$, competes with the damping mechanism for large values. Even for the hyperbolic Navier-Stokes equations, this circumstance poses a fundamental issue (see e.g.\ \cite{RackeSaal2012}).

\smallskip
\noindent
Secondly and much more inherent to the strategy followed in \cite{MR3925144} and Section \ref{sec:two}, the additional (second) time-derivatives produce corresponding terms terms on the right-hand side of the equation. Following Section \ref{sec:recasting-uk-into-psik}, we realize that two commutators  need to be evaluated in \eqref{firstCommutator} and \eqref{eq:d_ypsik-final-form}. In the classical Prandtl regime, the solenoidality of $(u,v)$  and cancellation of curl-related terms enters the analysis, e.g.
\begin{align*}
    [\partial_y, \partial_t +u \partial_x + v\partial_y -\partial_{yy}] u=0.
\end{align*}
None of these instances persist in the hyperbolic version. Instead, new terms arise like $\partial_{xt}v \sim \partial_{xxt} u$. Confronted with a $4$-order time derivative on the left-hand side, this does not present an improvement over the standard Gevrey 2 regularity result. A new,  different (perhaps related)  cancellation mechanism seems to be necessary but it is not clear how the system can be closed. 

\smallskip
\noindent
At this point, we remark that the independence of the shear flow $U_{\rm sh}$ on $(t,x)$ is exploited heavily in Section 4. A third variant of \eqref{simpleHP}, substituting the first equation by
\begin{equation*}
    (\tau  \partial_{tt} + \tau u\partial_x \partial_t + \eta v \partial_y \partial_t + \partial_t  +u \partial_x + v \partial_y - \partial_{yy} ) u =0,
\end{equation*}
bears similar problems, although the competing damping term $\partial_x u \partial_t u $ is not present here. However, the potential improvement of the above equation might lie in the transport structure related to the time-derivatives of $u$.

\smallskip
\noindent
In sum, we conjecture that Theorem \ref{main-thm}  hints  stability results around shear flows for the nonlinear hyperbolic equations in Gevrey 3 class. On the other hand, a general well-posedness theory for arbitrary initial data in $\mathcal{G}^3$ does not seem to be achievable without major novelties or improvements on the strategy followed in \cite{MR3925144} and this work.

\section*{Acknowledgment}

The author would like to thank Prof.~M.~Paicu for the several helpful advises on various technical issues examined in this Paper. The first author was partially supported by the Bavarian Funding Programme for the Initiation of International Projects (Förderkennzeichen: BayIntAn\textunderscore UWUE\textunderscore 2022\textunderscore 139). The third author was partially supported by GNAMPA and INDAM.

\subsection*{Data Availability Statement} 
Data sharing is not applicable to this article, since no datasets were generated or analysed during the current study.

\subsection*{Conflict of interest} 
The authors declare that they have no conflict of interest.

\printbibliography

\end{document}